\documentclass[12pt]{amsart}

\usepackage{graphicx} 
\usepackage{scrextend}
\usepackage{amsfonts, amsmath, amssymb,amsthm,comment}  
\usepackage{times,enumerate}
\usepackage{mathdots}%
\usepackage{nccmath}
\usepackage{mathtools}
\usepackage[shortlabels]{enumitem}
\usepackage{multicol}
\newenvironment{mpmatrix}{\begin{medsize}\begin{pmatrix}}%
    {\end{pmatrix}\end{medsize}}%
\usepackage{dsfont,bbm}
\usepackage{rotating}
\usepackage{lscape}
\usepackage{url}
\usepackage{multicol}
\usepackage{thmtools, thm-restate}
\usepackage{blkarray}
\usepackage{multicol}
\usepackage[margin=2.5cm]{geometry}
 \usepackage[foot]{amsaddr}
 \usepackage{hyperref}
\usepackage{cancel}

\usepackage[usenames, dvipsnames]{xcolor}
\DeclareGraphicsExtensions{.pdf,.png,.jpg}

\hypersetup{
    colorlinks,
    linkcolor={blue},
    citecolor={blue},
    urlcolor={blue}
}

\DeclareMathOperator{\Span}{span}

\DeclareMathOperator{\Ran}{Ran}
\DeclareMathOperator{\Rank}{rank}

\DeclareMathOperator{\Card}{card}

\DeclareMathOperator{\supp}{supp}
\DeclareMathOperator{\rev}{rev}
\DeclareMathOperator{\dom}{dom}

\makeatletter
\def\widebreve{\mathpalette\wide@breve}
\def\wide@breve#1#2{\sbox\z@{$#1#2$}%
     \mathop{\vbox{\m@th\ialign{##\crcr
\kern0.08em\brevefill#1{0.8\wd\z@}\crcr\noalign{\nointerlineskip}%
                    $\hss#1#2\hss$\crcr}}}\limits}
\def\brevefill#1#2{$\m@th\sbox\tw@{$#1($}%
  \hss\resizebox{#2}{\wd\tw@}{\rotatebox[origin=c]{90}{\upshape(}}\hss$}
\makeatletter

\newcommand{\RR}{\mathbb R}

\newcommand{\NN}{\mathbb N}
\newcommand{\ZZ}{\mathbb Z}
\newcommand{\CC}{\mathbb C}

\newcommand{\Mod}[1]{\ (\mathrm{mod}\ #1)}

\newcommand{\cB}{\mathcal B}

\newcommand{\cH}{\mathcal H}

\newcommand{\cZ}{\mathcal Z}

\newcommand{\cC}{\mathcal C}

\newcommand{\benu}{\begin{enumerate}}
\newcommand{\eenu}{\end{enumerate}}
\newcommand{\bop}{\begin{opomba}}
\newcommand{\eop}{\end{opomba}}

\newtheorem{theorem}{Theorem}[section]
\newtheorem{corollary}[theorem]{Corollary}

\newtheorem{proposition}[theorem]{Proposition}

\theoremstyle{definition}

\newcommand{\mc}{\mathcal}

\newcommand{\mbf}{\mathbf}

\theoremstyle{remark}
\newtheorem{remark}[theorem]{Remark}

\numberwithin{equation}{section}

\begin{document}

\title{The strong truncated Hamburger moment problem with and without gaps}

%----------Author 1
\author[Alja\v z Zalar]{Alja\v z Zalar}

\address{%
Faculty of Computer and Information Science\\
University of Ljubljana\\
Ve\v cna pot 113\\
1000 Ljubljana\\
Slovenia}

\email{aljaz.zalar@fri.uni-lj.si}

\thanks{Supported by the Slovenian Research Agency grants J1-2453, P1-0288.}


\date\today
\maketitle

\begin{abstract}
	The strong truncated Hamburger moment problem (STHMP) of degree $(-2k_1,2k_2)$ asks to find necessary 
	and sufficient conditions for the existence of a positive Borel measure, supported on $\RR$, such 
	that $\beta_i=\int x^id\mu\; (-2k_1\leq i\leq 2k_2)$. The first solution of the STHMP, covering also its
	matrix generalization, was established by Simonov \cite{Sim06}, who used the operator approach and 
	described all solutions in terms of self-adjoint extensions of a certain symmetric operator.
	Using the solution of the truncated Hamburger moment
	problem and the properties of Hankel matrices we give an alternative solution of the STHMP
	and describe concretely all minimal solutions, i.e., solutions having the smallest support.
	Then, using the equivalence with the STHMP of degree $(-2k,2k)$, 
	we obtain the solution of the 
	2--dimensional truncated moment problem (TMP) of degree $2k$  with variety $xy=1$, 
	first solved by Curto and Fialkow \cite{CF05}. Our addition to their result is the fact previously known only for 		
	$k=2$, that the existence of 
	a measure is equivalent to the existence of a flat extension of the moment matrix.
	Further on, we solve the STHMP of degree $(-2k_1,2k_2)$ with one missing moment in the sequence, 
	i.e., $\beta_{-2k_1+1}$ or $\beta_{2k_2-1}$, which also gives
	the solution of the TMP with variety $x^2y=1$ as a special case, first studied by Fialkow in \cite{Fia11}.
\end{abstract}




%%%%%%%%%%%%%%%%%%%%%%%%%%%%%%%%%%%%%%%%%%%%%%%%%%%%%%%%%%%%%%%%%%%%
%%%%%%%%%%%%%%%%%%%%%%%%%%%%%%%%%%%%%%%%%%%%%%%%%%%%%%%%%%%%%%%%%%%%
%%%%%%%%%%%%%%%%%%%%%%%%%%%%%%%%%%%%%%%%%%%%%%%%%%%%%%%%%%%%%%%%%%%%

\section{Introduction}

Given a real sequence $\beta^{ (-2k_1,2k_2)}=(\beta_{-2k_1},\beta_{-2k_1+1},\ldots,\beta_{2k_2-1},\beta_{2k_2})$ of degree $(-2k_1,2k_2)$, $k_1,k_2\in \ZZ_+$,
the \textbf{strong truncated Hamburger moment problem (STHMP)} for $\beta^{ (-2k_1,2k_2)}$
asks to characterize the existence of a positive Borel measure $\mu$ on $\RR$, such that
	\begin{equation}\label{moment-measure-cond}
		\beta_i=\int_{\RR}x^i d\mu\quad  (i\in \ZZ,\; -2k_1\leq i\leq 2k_2).
	\end{equation}
The STHMP of degree $(0,2k)$ is the usual \textbf{truncated Hamburger moment problem (THMP)} of degree $2k$.

We denote by $M(n_1,n_2)=M(n_1,n_2)(\beta^{ (-2k_1,2k_2)})=(\beta_{i+j})_{i,j=n_1}^{n_2}$, $-k_1\leq n_1\leq n_2\leq k_2$ 
the moment matrix associated with $\beta^{ (-2k_1,2k_2)}$, 
where the rows and columns are indexed by monomials $X^i$ in the degree increasing order 
	$$X^{n_1},X^{n_1+1},\ldots,X^{-1},1,X,\ldots,X^{n_2-1},X^{n_2}.$$
Let
	$\RR[x^{-1},x]_{r_1,r_2}=
		\left\{ \sum_{i=-r_1}^{r_2} a_i x^i \colon a_i \in \RR,\; r_1,r_2\in \ZZ_+\right\}$
	%$$\RR[x^{-1},x]_{k_1,k_2}:=\left\{p\in \RR[x^{-1},x]\colon p=\sum_{i=-k_1}^{-1}a_i(x^{-1})^i+\sum_{j=0}^{k_2} a_jx^j,\; a_{i}\in \RR\text{ for } -k_1\leq i\leq k_2\right\}$$ 
stand for the set of Laurent polynomials in variables $x^{-1},x$ of degree at most $r_1$ in $x^{-1}$ and at most 
$r_2$ in $x$.
For every Laurent polynomial 
$p(x^{-1},x)=\sum_{i=-k_1}^{k_2} a_i x^i\in\RR[x^{-1},x]_{k_1,k_2}$,
we denote by
	$$p(X^{-1},X)=\sum_{i=-k_1}^{-1}a_i(X^{-1})^i+a_0 1+ \sum_{j=1}^{k_2} a_jX^j\in 
		\cC_{M(-k_1,k_2)}$$
the vector from the column space $\cC_{M(-k_1,k_2)}$ of the moment matrix $M(-k_1,k_2)$. 
Let $\mbf{0}$ stand for the zero vector.
We say that the matrix $M(-k_1,k_2)$ is \textbf{recursively generated (rg)} if for $p,q,pq\in \RR[x^{-1},x]_{k_1,k_2}$ such that $p(X^{-1},X)=\mbf{0}$, 
it follows that $(pq)(X,X^{-1})=\mbf{0}.$ 

Given a real $2$--dimensional sequence
	$$\beta^{ (2k)}=\{\beta_{0,0},\beta_{1,0},\beta_{0,1},\ldots,\beta_{2k,0},\beta_{2k-1,1},\ldots,
		\beta_{1,2k-1},\beta_{0,2k}\}$$ of degree $2k$
and a closed subset $K$ of $\RR^2$, the \textbf{truncated moment problem (TMP)} supported on $K$ for $\beta^{ (2k)}$
asks to characterize the existence of a positive Borel measure $\mu$ on $\RR^2$ with support in $K$, such that
	\begin{equation}\label{moment-measure-cond}
		\beta_{i,j}=\int_{K}x^iy^j d\mu\quad (i,j\in \ZZ_+,\;0\leq i+j\leq 2k).
	\end{equation}
If such a measure exists, we say that $\beta^{ (2k)}$ has a representing measure supported on $K$ and $\mu$ is its $K$--\textbf{representing measure}.

We denote by $M(k)=M(k)(\beta^{ (2k)})=(\beta_{i,j})_{i,j=0}^{2k}$ the moment matrix associated with $\beta^{ (2k)}$, where the rows and columns are indexed 
in the degree lexicographic order 
	$$1,X,Y,\ldots, X^{2k},X^{2k-1}Y,\ldots, XY^{2k-1},Y^{2k}.$$
Let 
%$\RR[x,y]_k:=\{p\in \RR[x,y]\colon \deg p\leq k\}$ 
$\RR[x,y]_k$
stand for the set of polynomials in variables $x,y$ of degree at most $k$.
For every $p(x,y)=\sum_{i,j} a_{ij}x^iy^j\in \RR[x,y]_k$, we denote by $p(X,Y)=\sum_{i,j} a_{ij}X^iY^j$ the vector from the column
space $\cC_{M(k)}$ of the matrix $M(k)$. 
Recall from \cite{CF96}, that $\beta$ has a representing measure $\mu$ with the support $\supp \mu$ being a 
subset of 
$\cZ_p:=\{(x,y)\in \RR^2\colon p(x,y)=0\}$ if and only if $p(X,Y)=\mbf{0}$ where $\mbf{0}$ stands for the zero vector. 
We say that the matrix $M(k)$ is \textbf{recursively generated (rg)} if for $p,q,pq\in \RR[x,y]_k$ such that $p(X,Y)=\mbf{0}$, 
it follows that $(pq)(X,Y)=\mbf{0}.$
The \textbf{variety} of $\beta^{(2k)}$ is defined by 
$$\mathcal V(\beta^{(2k)}):=\bigcap_{
				\substack{g\in \RR[X,Y]_{\leq k},\\ 
							g(X,Y)=\mbf 0}} \mathcal Z_g,$$
where	$\mathcal Z_g:=\left\{ (x,y)\in \RR^2\colon g(x,y)=0 \right\}$.

A \textbf{concrete solution} to the TMP is a set of necessary and sufficient conditions for the existence of a $K$--representing measure, that can be tested in 
numerical examples. 
Among necessary conditions, $M(k)$ must be positive semidefinite (psd), rg and satisfies the \textbf{variety condition} \cite[Proposition 3.1 and Corollary 3.7]{CF96}, which
states that the inequality $\Rank M(k)\leq \Card \mathcal V(\beta^{(2k)})$ holds.
The celebrated \textbf{flat extension theorem} of Curto and Fialkow \cite[Theorem 7.10]{CF96}, \cite[Theorem 2.19]{CF05b} states that $\beta^{(2k)}$ admits a $\Rank M(k)$--atomic representing measure if and only if $M(k)$ is psd and admits a rank-preserving extension to a moment matrix $M(k+1)$. 
Using the flat extension theorem as the main tool the 2--dimensional TMP has been concretely solved in
the following cases: 
	$K$ is the variety defined by a polynomial $p(x,y)=0$ with $\deg p\leq 2$ 
	\cite{CF02, CF04, CF05,Fia14}, 
	$K=\RR^2$, $k=2$ and $M(2)$ is invertible \cite{CS16,FN10}, %\cite{FN10,CS16},
	$K$ is the variety $y=x^3$ \cite{Fia11},
% or \cite{Zal+}), $y^2=x^3$ \cite{Zal+} and $y=x^4$, $y^3=x^4$ with one additional moment given \cite{Zal+}, 
 	$M(k)$ has a special feature called \textit{recursive determinateness} \cite{CF13}
	and in the \textit{extremal case} with the equality in the variety condition \cite{CFM08}.
Some other special cases have been solved in \cite{Ble15,BF20,CS15,Fia17,Kim14}. %\cite{Kim14,CS15,Fia17,Ble15,BF20}. 
In \cite{Fia11}, Fialkow studied also the TMP for the curves of the form $y=g(x)$ and $yg(x)=1$, where $g\in \RR[x]$ is a polynomial, and obtained the
bound on the degree $m$ for which the existence of a positive extension $M(m)$ of $M(k)$ is equivalent to the existence of a measure.
In our previous work we derived some of the above results and solved new cases of the 2--dimensional TMP
using the solution of the THMP or the THMP with some missing moments: $K$ with variety $xy=0$ can
be solved with the use of the THMP twice \cite[Section 6]{BZ+}, $K$ with variety $y=x^3$ or
$y^2=x^3$ are equivalent to the THMP of degree $6k$ with a missing moment $\beta_{6k-1}$ or $\beta_1$ \cite[Subsections 3.1, 4.1]{Zal+}, while special cases of $K$ with variety $y=x^4$ or $y^3=x^4$ 
to the THMP of degree $8k$ without $\beta_{8k-2}$ and $\beta_{8k-1}$ or $\beta_{1}$ and $\beta_{2}$ 
\cite[Subsections 3.2, 4.2]{Zal+}.

By \cite{Sto01} the TMP is more general than the classical full moment problem (MP). For nice expositions on the full MP and the TMP see \cite{Akh65,KN77,Sch17}. 
Haviland's solution \cite{Hav35} of the MP established the duality of the MP with positive polynomials and
led to further investigations of the MP from the perspective of real algebraic geometry 
(see \cite{KP+,KM02,Las09,Lau05,Lau09,Mar08,Nie14,PS01,Put93,PV99,Sch91}). 
%(see \cite{Sch91, Put93,PV99,PS01,KM02,Lau05,Lau09,Mar08,Las09,Nie14,KP+}). 
Further on, various generalizations of the TMP and MP have been introduced, e.g., matrix and operator MPs
\cite{AV03,BW11,CZ12,HKM12,HM04,KW13,Kre49,McC01,Vas03}, 
%\cite{Kre49,McC01,AV03,Vas03,HM04,BW11,CZ12,HKM12,KW13}, 
tracial MPs 
\cite{BZ18,BZ+,BCKP13,BK10,BK12},  
%\cite{BK10,BK12,BCKP13,BZ18, BZ+}, 
MP supported on $\NN_0$ \cite{IKLS17}, MPs in infinitely many variables and on more general commutative algebras 
\cite{AJK15,CGIK+,DS18,GKM16}, 
%\cite{AJK15, GKM16,DS18,CGIK+}, 
TMP with a signed representing measure \cite{Kim+}.

In this article we first give an alternative solution to the STHMP of degree $(2k_1, 2k_2)$,
which was first solved in the more general matrix case in \cite{Sim06} using the operator approach,
describing all solutions in terms of self-adjoint extensions of a certain symmetric operator.
Our approach uses the solution of the THMP and the properties of Hankel matrices,
giving also a concrete description of all minimal solutions, i.e., solutions having the smallest support.
As a corollary we obtain a new proof of the TMP of degree $2k$ with variety $xy=1$, solved in \cite{CF05}. 
In addition, it follows that the existence of a flat extension of the moment matrix is equivalent to
the existence of a measure; for $k=2$ this was first proved in \cite[Proposition 5.3]{CF02}.
Then we solve the STHMP of degree $(-2k_1,2k_2)$ with the missing moment $\beta_{-2k_1+1}$ or $\beta_{2k_2-1}$ by using the solutions of the THMP of degree $2k$ with the missing moment $\beta_{1}$ or
$\beta_{2k-1}$ from \cite{Zal+}. Finally, as a corollary to this we obtain the solution of the TMP with variety $x^2y=1$.

%%%%%%%%%%%%%%%%%%%%%%%%%%%%%%%%%%%%%%%%%%%%%%%%%%%%%%%%%%%%%%%%%%%%

\subsection{ Reader's Guide}
The paper is organized as follows. 
In Section \ref{S2} we present some properties of psd Hankel matrices %(Subsection \ref{SubS2.2}) 
and the solution of the THMP. %(Subsection \ref{SubS2.3}). 
%We replace the variety condition
%in solution by the condition on ranks of certain moment submatrices (see Proposition \ref{equiv-0110-1306}).
In Section \ref{S3} we first state the solution of the STHMP (see Theorem \ref{strongHamburger-general}), give a proof
based on the solution of the THMP in Subsection \ref{sub-sec-3-1}, explain the connection with Simonov's approach \cite{Sim06} in Subsection \ref{140222-2346},
and finally as a corollary
obtain the solution of the nondegenerate hyperbolic TMP (see Corollary \ref{posledica-12:51}).
Finally, in Section \ref{S4} we present the solutions of the STHMP of degree $(-2k_1,2k_2)$ with
the missing moment $\beta_{-2k_1+1}$  (see Theorem \ref{strong-trunc-Hamb-without-(-2k+1)}) 
or $\beta_{-2k_2-1}$  (see Corollary \ref{strong-trunc-Hamb-without-(2k-1)}) and as a consequence solve
 the TMP for the variety $x^2y=1$ (see Corollary \ref{X2Y=-general}).

%%%%%%%%%%%%%%%%%%%%%%%%%%%%%%%%%%%%%%%%%%%%%%%%%%%%%%%%%%%%%%%%%%%%

\noindent \textbf{Acknowledgement}.  
I would like to thank the anonymous referee for very useful comments for the improvement of
the manuscript and bringing the paper \cite{Sim06} to my attention.

%%%%%%%%%%%%%%%%%%%%%%%%%%%%%%%%%%%%%%%%%%%%%%%%%%%%%%%%%%%%%%%%%%%%
%%%%%%%%%%%%%%%%%%%%%%%%%%%%%%%%%%%%%%%%%%%%%%%%%%%%%%%%%%%%%%%%%%%%

\section{Preliminaries}\label{S2}
We write $M_{n,m}$ (resp.\ $M_n$) for the set of $n\times m$ (resp.\ $n\times n$) real matrices. 
For a matrix $M$ we denote by $\cC_M$ its column space.
The set of real symmetric matrices of size $n$ will be denoted by $S_n$. 
For a matrix $A\in S_n$ the notation $A\succ 0$ (resp.\ $A\succeq 0$) means $A$ is positive definite (pd) (resp.\ positive semidefinite (psd)).

%%%%%%%%%%%%%%%%%%%%%%%%%%%%%%%%%%%%%%%%%%%%%%%%%%%%%%%%%%%%%%%%%%%%

%\subsection{Hankel matrices}\label{SubS2.2}
Let $k\in \NN$.
For 
	\begin{equation*}\label{vector-v}
		v=(v_0,\ldots,v_{2k} )\in \RR^{2k+1},
	\end{equation*}
we denote by
	\begin{equation*}\label{vector-v}
		A_{v}:=\left(v_{i+j} \right)_{i,j=0}^k
					=\left(\begin{array}{ccccc} 
							v_0 & v_1 & v_2 & \cdots & v_k\\
							v_1 & v_2 & \iddots & \iddots & v_{k+1}\\
							v_2 & \iddots & \iddots & \iddots & \vdots\\
							\vdots 	& \iddots & \iddots & \iddots & v_{2k-1}\\
							v_k & v_{k+1} & \cdots & v_{2k-1} & v_{2k}
						\end{array}\right)
					\in S_{k+1}
	\end{equation*}
the corresponding Hankel matrix. We denote by 
	$\mbf{v_j}:=\left( v_{j+\ell} \right)_{\ell=0}^k$ the $(j+1)$--th column of $A_{v}$, $0\leq j\leq k$, i.e.,
		$$A_{v}=\left(\begin{array}{ccc} 
								\mbf{v_0} & \cdots & \mbf{v_k}
							\end{array}\right).$$
As in \cite{CF91}, the \textbf{rank} of $v$, denoted by $\Rank v$, is defined by
	$$\Rank v=
	\left\{\begin{array}{rl} 
		k+1,&	\text{if } A_{v} \text{ is nonsingular},\\
		\min\left\{i\colon \bf{v_i}\in \Span\{\bf{v_0},\ldots,\bf{v_{i-1}}\}\right\},&		\text{if } A_{v} \text{ is singular}.
	 \end{array}\right.$$
If $\Rank v<k+1$ we say that $v$ is \textbf{singular}. Else $v$ is \textbf{nonsingular}.

We denote
\begin{itemize} 
	\item the upper left-hand corner $\left(v_{i+j} \right)_{i,j=0}^m\in S_{m+1}$ of $A_{v}$ of size $m+1$ by 
		$A_{v}(m)$.
	\item the lower right-hand corner $\left(v_{i+j} \right)_{i,j=k-m}^k\in S_{m+1}$ of $A_{v}$ of size $m+1$ by
		$A_{v}[m]$.
\end{itemize}

For a sequence $v=(v_0,\ldots,v_{2k})$ we denote by $v^{(\rev)}:=(v_{2k},v_{2k-1},\ldots,v_0)$ the \textbf{reversed sequence}. 
A sequence $v$ is called
\begin{itemize}
	\item \textbf{positively recursively generated (prg)} if for  $r=\Rank v$ the following two conditions hold:
		\begin{itemize}
			\item $A_v(r-1)\succ 0$.
			\item If $r<k+1$, denoting 
				\begin{equation}\label{rg-coefficients-2809-2000}
					(\varphi_0,\ldots,\varphi_{r-1}):=A_{v}(r-1)^{-1}(v_r,\ldots,v_{2r-1})^{T},
				\end{equation}
				the equality
				\begin{equation}\label{recursive-generation}
					  v_j=\varphi_0v_{j-r}+\cdots+\varphi_{r-1}v_{j-1}%\quad \text{for }j=r,\ldots,2k.
				\end{equation}
				holds for $j=r,\ldots,2k$.
		 \end{itemize}
	\item \textbf{negatively recursively generated (nrg)}  if for $r=\Rank v^{(\rev)}$ the following two conditions hold:
		\begin{itemize}
			\item $A_v[r-1]\succ 0$.
			\item If $r<k+1$, denoting
					$$(\psi_0,\ldots,\psi_{r-1}):=A_{v}[r-1]^{-1}(v_{2k-2r+1},\ldots,v_{2k-r})^{T},$$
				the equality
				\begin{equation}\label{negative-recursive-generation}
			  		v_{2k-r-j}=\psi_0v_{2k-r+1-j}+\cdots+\psi_{r-1}v_{2k-j},
				\end{equation}
				holds for $j=0,\ldots,2k-r$.
		\end{itemize}
	\item \textbf{recursively generated (rg)} if it is prg and nrg,
\end{itemize}

\begin{proposition} \label{prg-2809-1955}
   Let $v=(v_0,\ldots,v_{2k})\in \RR^{2k+1}$, $v_0>0$, be a singular sequence of rank $r\leq k$ such that $A_v\succeq 0$. Let
   $\varphi_i$ be defined by \eqref{rg-coefficients-2809-2000}. Then the following statements are true:
	\begin{enumerate}[(1)]
		\item\label{pt0-prg-2809-2306}  \eqref{recursive-generation} holds for $j=r,\ldots,2k-1.$
		\item\label{pt0-prg-0110-1106}  \eqref{negative-recursive-generation} holds for $j=0,\ldots,2k-r-1.$		
		\item\label{pt-1-prg-2809-2310} The polynomial $\displaystyle p(x):=x^r-\sum_{i=0}^{r-1}\varphi_i x^i$ has $r$ distinct real zeroes.
		\item\label{pt1-prg-2809-1956} The following statements are equivalent:
			\begin{enumerate}
				\item $v$ is prg.
				\item \eqref{recursive-generation} holds for $j=2k$.
				\item $\Rank A_v(k-1)=\Rank A_v$.
				\item There exist real numbers $v_{2k+1}$ and $v_{2k+2}$ such that 
					$A_{\tilde v}\succeq 0$, where $\tilde v:=(v,v_{2k+1},v_{2k+2})$.
				\item $v_{2k+1}$ and $v_{2k+2}$ defined by \eqref{recursive-generation} for $j=2k+1, 2k+2$ are the unique real numbers such that 
					$A_{\tilde v}\succeq 0$ and $\Rank A_v=\Rank A_{\tilde v}$,
					where $\tilde v:=(v,v_{2k+1},v_{2k+2})$.
			\end{enumerate}
		\item\label{pt2-nrg-2809-1956} The following statements are equivalent:
			\begin{enumerate}
				\item $v$ is nrg.
				\item \eqref{negative-recursive-generation} holds for $j=2k-r+1$.
				\item $\Rank A_v[k-1]=\Rank A_v$.
				\item There exist real numbers $v_{-1}$ and $v_{-2}$ such that 
					$A_{\tilde v}\succeq 0$, where $\tilde v:=(v_{-2},v_{-1},v)$.
				\item $v_{-2}$ and $v_{-1}$ defined by \eqref{negative-recursive-generation} for $j=2k-r+1, 2k-r+2$ are the unique real numbers such that 
					$A_{\tilde v}\succeq 0$ and $\Rank A_v=\Rank A_{\tilde v}$,
					where $\tilde v:=(v_{-2},v_{-1},v)$.
			\end{enumerate}
	\end{enumerate}
\end{proposition}

\begin{proof}
	\ref{pt0-prg-2809-2306} is \cite[Theorem 2.4(ii)]{CF91}. \ref{pt-1-prg-2809-2310} follows from \cite[Remark 3.5]{CF91}.
	\ref{pt1-prg-2809-1956} follows from \cite[Theorem 2.6 and Remark 2.7]{CF91}.
	Using \ref{pt0-prg-2809-2306} (resp.\ \ref{pt1-prg-2809-1956}) for $v^{(\rev)}$ we obtain \ref{pt0-prg-0110-1106} (resp.\ \ref{pt2-nrg-2809-1956}).
\end{proof}

\begin{remark}
\begin{enumerate}
	\item Proposition \ref{prg-2809-1955}.\ref{pt0-prg-2809-2306} implies that for a singular sequence $v$, the numbers $\varphi_i$ could also be defined as
		the unique coefficients such that  $\mbf{v_r}=\varphi_0 \mbf{v_{0}}+\cdots+\varphi_{r-1}\mbf{v_{r-1}}$. Moreover,
		\begin{equation}\label{recursive-generation-equivelant}
		  \mbf{v_j}=\varphi_0 \mbf{v_{j-r}}+\cdots+\varphi_{r-1}\mbf{v_{j-1}} %\quad \text{for}\quad j=r,\ldots,k,
		\end{equation}
		holds for $j=r+1,\ldots,k-1$.
	\item Proposition \ref{prg-2809-1955}.\ref{pt1-prg-2809-1956} implies that  $v$ is prg if and only if \eqref{recursive-generation-equivelant} holds also for $j=k$.
	\item Proposition \ref{prg-2809-1955}.\ref{pt0-prg-0110-1106} implies that for a singular sequence $v$, the numbers $\psi_i$ could also be defined as
		the unique coefficients such that  $\mbf{v_{k-r}}=\psi_0 \mbf{v_{k-r+1}}+\cdots+\psi_{r-1}\mbf{v_{k}}$. Moreover,
		\begin{equation}\label{negative-recursive-generation-equivelant}
		  \mbf{v_{k-r-j}}=\psi_0 \mbf{v_{k-r+1-j}}+\cdots+\psi_{r-1}\mbf{v_{k-j}} %\quad \text{for}\quad j=0,\ldots,k-r,
		\end{equation}
		holds for $j=1,\ldots,k-r-1$.
	\item Proposition \ref{prg-2809-1955}.\ref{pt2-nrg-2809-1956}  implies that $v$ is nrg if and only if \eqref{negative-recursive-generation-equivelant} holds also for $j=k-r$.
\end{enumerate}
\end{remark}

Let $v=(v_0,\ldots,v_{2k})\in \RR^{2k+1}$ be a sequence with the Hankel matrix $A_{v}=\left(\begin{array}{ccc} \mbf{v_0} & \cdots & \mbf{v_k}\end{array}\right)$.
For a polynomial $g(x)=\sum_{i=0}^k\gamma_i x^i$, $\gamma_i\in \RR$, we define the \textbf{evaluation} $g(v)$ by the rule
$g(v)=\sum_{i=0}^k\gamma_i \mbf{v}_i.$
For a singular sequence $v$ we call the polynomial $p$ from Proposition \ref{prg-2809-1955}.\ref{pt-1-prg-2809-2310}
the \textbf{generating polynomial} of $v$. 
%If \eqref{recursive-generation-equivelant} holds also for $j=k$, $p$ is called \textbf{strictly generating polynomial} of $v$. 
%For a polynomial $g(x)=\sum_{i=0}^k\gamma_i x^i$, $\gamma_i\in \RR$, a sequence $v=(v_0,\ldots,v_{2k})$, 
We write $\mbf{0}\in \RR^{k+1}$ for the zero vector. 
%
%\begin{proposition}\label{rg-for-Hankel-0210-1819}
%  Let $v=(v_0,\ldots,v_{2k})\in \RR^{2k+1}$, $v_0>0$, be a singular sequence of rank $r\leq k$ such that $A_v\succeq 0$. 
%  %Let  $\varphi_i$ be defined by \ref{rg-coefficients-2809-2000} and $A_v=\begin{mpmatrix}\mbf{v_0} & \cdots & \mbf{v_k}\end{mpmatrix}$.
%   Let  $g(x)$ be a nonzero polynomial in $x$ such that $g(v)=\mbf{0}$. Then the polynomial $g(x)$ is divisible by the generating polynomial of $v$.
%\end{proposition}
%
%\begin{proof}
%	Let $\ell$ be the degree of $g$. If $\ell<r$, then $\Rank v<r$, which is a contradiction. Dividing the polynomial $g$ by $p$ we get
%	$g=pp_1+p_2$, where $p_1$ and $p_2$ are polynomials of degrees $\ell-r$ and $\ell_2<r$, respectively.
%	Since \eqref{recursive-generation-equivelant} holds for $j=r,\ldots,k-1$ and if $\ell=k$, by Proposition  \ref{prg-2809-1955}.\ref{pt1-prg-2809-1956},
%	also for $k$, it follows that $(pp_1)(v)=\mbf{0}$. Hence, $p_2(v)=\mbf{0}$, which is possible only if $p_2$ is trivial.
%\end{proof}

\begin{proposition} \label{prg-nrg-rg-2809-1419}
   For a singular sequence $v=(v_0,\ldots,v_{2k})\in \RR^{2k+1}$ the following statements are equivalent:
	\begin{enumerate}[(1)]
		\item\label{pt1-prg-nrg-2809-1420} $v$ is prg and $\varphi_0\neq 0$.
		\item\label{pt2-prg-nrg-2809-1420} $v$ is nrg and $\psi_{r-1}\neq 0$.
		\item\label{pt3-prg-nrg-2809-1421} $v$ is rg.
		\item\label{pt4-prg-nrg-2809-1424} $v$ is rg, $\Rank v=\Rank v^{(\rev)}$, $\varphi_0\neq 0$ and
			\begin{equation}\label{psi-varphi-2809-1922}
				(\psi_0,\psi_1,\ldots,\psi_{r-2},\psi_{r-1})=
				\left(-\frac{\varphi_{1}}{\varphi_0},-\frac{\varphi_{2}}{\varphi_0},\ldots,-\frac{\varphi_{r-1}}{\varphi_0},\frac{1}{\varphi_0}\right).
			\end{equation}
	\end{enumerate}
\end{proposition}

\begin{proof}
		First we prove the implication $\ref{pt1-prg-nrg-2809-1420}\Rightarrow\ref{pt2-prg-nrg-2809-1420}$.
	By definition of $\Rank v=r$, the set $\{\mbf{v}_0,\ldots,\mbf{v}_{r-1}\}$ is linearly independent.
	Since $v$ is prg, Proposition \eqref{prg-2809-1955}.\ref{pt1-prg-2809-1956} implies that \eqref{recursive-generation-equivelant} holds for $j=r,\ldots,k$.
	Since $\varphi_0\neq 0$, it follows that for $i=0,\ldots,k-r$ 
		$$\mbf{v}_i=-\sum_{j=1}^{r-1}\frac{\varphi_{j}}{\varphi_0}\mbf{v}_{i+j}+\frac{1}{\varphi_0} \mbf{v}_{i+r},$$
	and inductively every set $\{\mbf{v}_{i+1},\ldots,\mbf{v}_{i+r}\}$, $i=0,\ldots, k-r$, is linearly independent.
	Hence, \eqref{negative-recursive-generation-equivelant} is true with $\psi_j=-\frac{\varphi_{j+1}}{\varphi_0}$, $j=0,\ldots,r-2$, and 
	$\psi_{r-1}=\frac{1}{\varphi_0}$.

		The proof of the implication $\ref{pt1-prg-nrg-2809-1420}\Leftarrow\ref{pt2-prg-nrg-2809-1420}$ is analoguous to the proof
	of $\ref{pt1-prg-nrg-2809-1420}\Rightarrow\ref{pt2-prg-nrg-2809-1420}$ only that the induction step is in the backward direction.

		Since $\ref{pt1-prg-nrg-2809-1420}$ and $\ref{pt2-prg-nrg-2809-1420}$ are equivalent, the implication 
	$\ref{pt1-prg-nrg-2809-1420}\Rightarrow\ref{pt3-prg-nrg-2809-1421}$ follows. The nontrivial part of the implication 
	$\ref{pt1-prg-nrg-2809-1420}\Leftarrow \ref{pt3-prg-nrg-2809-1421}$ is $\varphi_0\neq 0$.
	If $\varphi_0=0$, then since \eqref{recursive-generation-equivelant} holds for $j=r,\ldots, k$ and 
	$\mbf{v}_{0},\ldots,\mbf{v}_{r-1}$ are linearly independent, it follows that $\mbf{v}_0\notin \Span\{\mbf{v}_1,\ldots,\mbf{v}_{k}\}$.
	Since $v$ is singular, $A_v$ is also singular, and consequently $A_{v^{(\rev)}}$ and $v^{(\rev)}$ are both singular. 
	Since $v$ is nrg, it follows from \eqref{negative-recursive-generation-equivelant}, used for $j=k-r$, that
	$\mbf{v}_0\in \Span\{\mbf{v}_1,\ldots,\mbf{v}_{k}\}$, which is a contradiction. Hence, $\varphi_0\neq 0$.
	
	The nontrivial implication of the equivalence $\ref{pt1-prg-nrg-2809-1420}\Leftrightarrow\ref{pt4-prg-nrg-2809-1424}$ is 
	$\ref{pt1-prg-nrg-2809-1420}\Rightarrow\ref{pt4-prg-nrg-2809-1424}.$ Note that the equalities $\Rank v=\Rank v^{(\rev)}$ and 
	\eqref{psi-varphi-2809-1922} follow from the proof of the implication $\ref{pt1-prg-nrg-2809-1420}\Rightarrow\ref{pt2-prg-nrg-2809-1420}$.
\end{proof}

%We will need the following proposition on psd extensions of a given matrix, 
%originally proved in \cite{Smu59}, 
%when studying rank-preserving extensions of Hankel matrices.
%
%\begin{proposition}\cite[Lemma 2.3]{CF91}\label{rank-13-07}
%	Let 
%		\begin{equation*}%\label{form-of-M-2}
%			\widetilde{A}=\left( \begin{array}{cc} A & b \\ b^{T} & c\end{array}\right)\in S_{n+1}
%		\end{equation*} 
%	be a real symmetric matrix where $A\in S_n$, $b\in \RR^n$ and $c\in \RR$.	
%	Assume that $A\succeq 0$ and $b=Aw$ for some $w\in \RR^n$.
%	Then:
%	\begin{enumerate}[(1)]
%		\item \label{pt2-0110-155} $\widetilde A\succeq 0$ if and only if $c\geq w^T A w$.
%		\item\label{pt1-0110-155} $\Rank \widetilde A=\Rank A$ if and only if $c=w^T A w$.
%		\item\label{pt3-0110-155} $\widetilde A\succ 0$ if and only if $c>w^T A w$.
%	\end{enumerate}
%\end{proposition}

%We will also need the following extension principle for psd matrices.
%
%\begin{lemma}\cite[Lemma 2.12]{Zal+}\label{extension-principle}
%	Let $A\in S_n$ be a positive semidefinite matrix, 
%	$Q\subseteq \{1,\ldots,n\}$ a subset 
%	and	
%	$A_Q$ be the restriction of $A$ to rows and columns from the set $Q$. 
%	If $v\in \ker A_Q$ is a nonzero vector from the kernel of $A_Q$,
%	then the vector $\widehat{v}$ with the only nonzero entries in rows from $Q$ and such that the restriction 
%	$\widehat{v}|_Q$ to the rows from $Q$ equals to $v$, belongs to $\ker A$.  
%\end{lemma}

%%%%%%%%%%%%%%%%%%%%%%%%%%%%%%%%%%%%%%%%%%%%%%%%%%%%%%%%%%%%%%%%%%%%

%\subsection{The truncated Hamburger moment problem}\label{SubS2.3}

For $x\in \RR^m$ we use $\delta_x$ to denote the probability measure on $\RR^m$ such that $\delta_x(\{x\})=1$. 
By a \textbf{finitely atomic positive measure} on $\RR^m$ we mean a measure of the form $\mu=\sum_{j=0}^\ell \rho_j \delta_{x_j}$, 
where $\ell\in \NN$, each $\rho_j>0$ and each $x_j\in \RR^m$. The points $x_j$ are called 
\textbf{atoms} of the measure $\mu$ and the constants $\rho_j$ the corresponding \textbf{densities}.

For $v:=(v_1,\ldots, v_{m})\in \RR^m$ we denote by $V_v\in \RR^{m\times m}$ the Vandermondo matrix 
	$$V_v:=
		\left(\begin{array}{cccc}
		1 & 1 & \cdots & 1\\
		v_1 & v_2 & \cdots & v_m\\
		\vdots & \vdots &  & \vdots\\
		v_1^{m-1} & v_2^{m-1} & \cdots & v_m^{m-1}
		\end{array}\right).$$ 

The solution of the THMP of degree $2k$ is the following.

\begin{theorem}[{\cite[Theorems 3.9 and 3.10]{CF91}}]\label{Hamburger}
	For $k\in \NN$ and $\beta=(\beta_0,\ldots,\beta_{2k})\in \RR^{2k+1}$ with $\beta_0>0$, the following statements are equivalent:
\begin{enumerate}[(1)]	
	\item\label{pt1-130222-1851} There exists a $\RR$--representing measure for $\beta$, i.e., supported on $\RR$.
	\item There exists a $(\Rank \beta)$--atomic representing measure for $\beta$.
	\item\label{pt3-130222-1851} $\beta$ is positively recursively generated.
	\item $M(0,k)\succeq 0$ and $\Rank M(0,k)=\Rank \beta$.
	\item\label{pt4-v2206} One of the following statements holds:
	\begin{enumerate}
		\item $M(0,k)\succ 0$.
		\item $M(0,k)\succeq 0$ and $\Rank M(0,k)=\Rank M(0,k-1)$.
	 \end{enumerate}
\end{enumerate}

\begin{enumerate}[(i)]
	\item\label{140222-1158}
$r\leq k$, then the $\RR$--representing measure $\mu$ is unique and of the form
	$\mu=\sum_{i=1}^{r}\rho_i\delta_{x_i},$ where $x_1,\ldots,x_r$ are the roots of the 
	generating polynomial of $\beta$,
		$$\left(\begin{array}{ccc}\rho_1 & \cdots &\rho_{r}\end{array}\right)^{T}:=V_x^{-1}u,$$
	$x=(x_1,\ldots,x_{r})$ and $u=\left(\begin{array}{ccc}\beta_0 & \cdots &\beta_{r-1}\end{array}\right)^{T}$.
	\item\label{140222-1202} $r=k+1$, then there are infinitely many $\RR$--representing measures for $\beta$. All $(k+1)$--atomic ones
		are obtained by choosing $\beta_{2k+1}\in \RR$ arbitrarily, defining 
		$\beta_{2k+2}:=u^T{(M(0,k))}^{-1}u$, where  
		$u=\left(\begin{array}{ccc}\beta_{k+1} & \cdots &\beta_{2k+1}\end{array}\right)^{T}$,
		and use \ref{140222-1158} for $\widetilde \beta:=(\beta_0,\ldots,\beta_{2k+1},\beta_{2k+2})\in \RR^{2k+3}.$ 
\end{enumerate}
\end{theorem}

For a vector $v\in \RR^{m}$ we denote by $v(0\mathbin{:}i)\in \RR^{i+1}$ the projection on the first $i+1$ coordinates 
and by $v(i):=v(i\mathbin{:}i)\in \RR$ the $(i+1)$--th coordinate of $v$.

We will need the following proposition in the solution of the STHMP. 

\begin{proposition}\label{140222-1130}
	Let $k\in \NN$ and $\beta=(\beta_0,\ldots,\beta_{2k})\in \RR^{2k+1}$ with $\beta_0>0$ be a real sequence such that $M(0,k)\succ 0$.
	Then: 
	\begin{enumerate}[(1)]
	\item\label{140222-1441} All but at most one $(k+1)$--atomic representing measures for $\beta$ described in Theorem \ref{Hamburger}.\ref{140222-1202} are supported on $\RR\setminus \{0\}$
		and the corresponding sequences $\widetilde \beta$ are singular and recursively generated. 
	\item\label{140222-1442} Denoting $M(0,k)=\left(\begin{array}{cccc} \mbf{v_0} & \mbf{v_1} & \cdots & \mbf{v_k} \end{array}\right)$ the $(k+1)$--atomic representing measure for $\beta$ with a nonzero density in $0$
		exists if and only if 
			$$C:=\left(\begin{array}{ccc} \mbf{v_1}(0\mathbin{:} k-1) & \cdots & \mbf{v_k}(0\mathbin{:} k-1) \end{array}\right)$$ 
		is invertible. In this case
		$\beta_{2k+1}=w^TC^{-1}w$, where $w=\left(\begin{array}{ccc} \beta_{k+1} & \cdots & \beta_{2k} \end{array}\right)^T$ and $\beta_{2k+2}$ is
		as in Theorem \ref{Hamburger}.\ref{140222-1202}.
	\end{enumerate}
\end{proposition}

\begin{proof}
	Let $\beta_{2k+1}\in \RR$ be arbitrary and $\widetilde \beta$ be defined as in Theorem \ref{Hamburger}.\ref{140222-1202}.
	By \cite[Lemma 2.3]{CF91} 
	%Proposition \ref{rank-13-07} 
	we have that $\Rank A_{\widetilde \beta}=\Rank M(0,k)$ and hence $\widetilde\beta$ is singular.
	By Theorem \ref{Hamburger}.\ref{140222-1158}, $\widetilde \beta$ has a unique $(k+1)$--atomic representing measure supported on the set of roots $\cZ(p_{\widetilde\beta})$ of the generating polynomial $p_{\widetilde\beta}$ of $\widetilde\beta$.
	To establish \ref{140222-1441} it remains to prove that for all but one $\beta_{2k+1}$, $\cZ(p_{\widetilde\beta})$ does not contain 0 and $\widetilde \beta$ is rg.
	We write
		$A_{\widetilde \beta}
		=\left(\begin{array}{ccccc} \mbf{u_0} & \mbf{u_1} & \cdots & \mbf{u_k} & \mbf{u_{k+1}} \end{array}\right).$
	Assume that $\cZ(p_{\widetilde\beta})$ contains 0. Then $p_{\widetilde\beta}(x)=x^{k+1}-\sum_{i=1}^{k}\varphi_i x^i$ for some $\varphi_i\in \RR$ or equivalently
	$\mbf{u_{k+1}}=\sum_{i=1}^{k}\varphi_i \mbf{u_{i}}$.
	In particular, 
	\begin{equation}\label{140222-1445} 
		\mbf{u_{k+1}}(0\mathbin{:}k-1)=\sum_{i=1}^{k}\varphi_i \mbf{u_{i}}(0\mathbin{:}k-1)
	\end{equation} 
	and $\beta_{2k+1}=\mbf{u_{k+1}}(k)=\sum_{i=1}^{k}\varphi_i \mbf{u_{i}}(k)$.
	If the vectors $\mbf{u_{1}}(0\mathbin{:}k-1), \ldots, \mbf{u_{k}}(0\mathbin{:}k-1)$ are linearly independent, then $\varphi_1,\ldots,\varphi_k$ satisfying \eqref{140222-1445} are uniquely determined
	and hence also $\beta_{2k+1}$, such that $\cZ(p_{\widetilde\beta})$ contains $0$, is unique. Otherwise $\mbf{u_{1}}(0\mathbin{:}k-1), \ldots, \mbf{u_{k}}(0\mathbin{:}k-1)$ are linearly dependent and 
	thus
	\begin{align}\label{170222-1215}
	\begin{split}
		k
		&>\Rank \left(\begin{array}{ccccc} \mbf{u_1}(0\mathbin{:}k-1) & \mbf{u_2}(0\mathbin{:}k-1) & \cdots & \mbf{u_{k+1}}(0\mathbin{:}k-1) \end{array}\right)\\
		&=\Rank \left(\begin{array}{ccccc} \mbf{u_1}(0\mathbin{:}k-1) & \mbf{u_2}(0\mathbin{:}k-1) & \cdots & \mbf{u_{k+1}}(0\mathbin{:}k-1) \end{array}\right)^T\\
		&=\Rank \left(\begin{array}{c} (\mbf{u_1}(0\mathbin{:}k-1))^T \\ M(1,k)\end{array}\right)=k,
	\end{split}
	\end{align}
	where we used the Hankel structure of $A_{\widetilde \beta}$ in the second equality and $M(1,k)\succ 0$ in the third equality. \eqref{170222-1215} is a contradiction. Thus there is at most one
	$\beta_{2k+1}$ such that $\cZ(p_{\widetilde\beta})$ contains 0. If $\cZ(p_{\widetilde\beta})$ does not contain 0, then
	 $p_{\widetilde\beta}(x)=x^{k+1}-\sum_{i=0}^{k}\varphi_i x^i$ with $\varphi_0\neq 0$. By Proposition \ref{prg-nrg-rg-2809-1419}, $\widetilde\beta$ is rg in this case. This proves \ref{140222-1441}.
	The statement
	\ref{140222-1442} also follows from the proof of \ref{140222-1441} above by noticing that $\mbf{u_i}(0\mathbin{:}k)=\mbf{v_i}(0\mathbin{:}k)$ for $i=0,\ldots,k$ and $\mbf{u_{k+1}}(0\mathbin{:}k-1)=w$.
\end{proof}

\section{The STHMP and the TMP with variety $xy=1$}
\label{S3}

In this section we first solve the STHMP (see Theorem \ref{strongHamburger-general}) and then as a corollary obtain the solution of the TMP for the curve $xy=1$ (see Corollary \ref{posledica-12:51}). 

\begin{theorem}\label{strongHamburger-general}
	For $k_1,k_2\in \NN$, let $\beta:=\beta^{(-2k_1,2k_2)}=
	(\beta_{-2k_1},\beta_{-2k_1+1},\ldots\beta_{2k_2})$ be a real sequence of degree $(-2k_1,2k_2)$,
	such that $\beta_{-2k_1}>0$, with the associated moment matrix $M(-k_1,k_2)$. The following statements are equivalent:
\begin{enumerate}[(1)]	
	\item\label{pt1-2609-1306} There exists a representing measure for $\beta$ supported on $\RR\setminus\{0\}$.
	\item\label{pt2-2609-1309} There exists a $(\Rank \beta)$--atomic representing measure for $\beta$ supported on $\RR\setminus\{0\}$.
	\item\label{rg-21-12-1158} $\beta$ is recursively generated.
	\item\label{rankcard-29-09-1325}  $M(-k_1,k_2)\succeq 0$ and one of the following statements holds:
	\begin{enumerate}[(a)]
		\item\label{1712-2257} $M(-k_1,k_2)\succ 0$.
		\item\label{pt2-0310-0958} $\Rank M(-k_1,k_2)=\Rank M(-k_1,k_2-1)=\Rank M(-k_1+1,k_2).$
	\end{enumerate}
\end{enumerate}

Moreover, if $\beta$ with $r=\Rank \beta$ has a $(\RR\setminus\{0\})$--representing measure and:
\begin{enumerate}[(i)]
	\item\label{140222-1407} $r\leq k_1+k_2$, then the representing measure is unique and of the form
		$\mu=\sum_{i=1}^{r}\rho_i\delta_{x_i},$ where $x_1,\ldots,x_r$ are the roots of the 
		generating polynomial of $\beta$ and $\rho_1,\ldots,\rho_r>0$ the corresponding densities.
	\item\label{140222-1408} $r=k_1+k_2+1$, then there are infinitely many $(k_1+k_2+1)$--atomic representing measures for $\beta$.
		Denoting $M(-k_1,k_2)=\left(\begin{array}{cccc} \mbf{v_0} & \mbf{v_1} & \cdots & \mbf{v_{k_1+k_2}} \end{array}\right)$,
		they are obtained by choosing any $\beta_{2k_2+1}\in \RR$, 
		which is not equal to $v^TC^{-1}v$ if $C$ is invertible, where 	
			$$C=\left(\begin{array}{ccc} \mbf{v_1}(0\mathbin{:}k_1+k_2-1) & \cdots & \mbf{v_k}(0\mathbin{:}k_1+k_2-1) \end{array}\right)$$
		and $v=\left(\begin{array}{cccc} \beta_{-k_1+k_2+1} & \cdots & \beta_{2k_2}\end{array}\right)^T$,
		defining 
			$$\beta_{2k_2+2}=u^T\left(M(-k_1,k_2)\right)^{-1}u,$$
		where $u=\left(\begin{array}{cccc} \beta_{-k_1+k_2+1}&\cdots&\beta_{2k_2}&\beta_{2k_2+1}\end{array}\right)^T$,
		and then use \ref{140222-1407} for $$\widetilde\beta=(\beta_{-2k_1},\ldots,\beta_{2k_2+1},\beta_{2k_2+2}).$$
\end{enumerate}
\end{theorem}

\begin{remark}
Before proving Theorem \ref{strongHamburger-general} let us mention that the matrix STHMP was already considered by Simonov in \cite{Sim06}.
Let $N\in \NN$ and $H_N(\CC)$ be the set of $N\times N$ complex hermitian matrices.
The matrix STHMP of degree $(-2k_1,2k_2)$, $k_1,k_2\in \ZZ_+$ refers to the case when $\{S_i\}_{i=-2k_1}^{2k_2}$ is a sequence of hermitian $N\times N$ complex matrices and one wants to find all positive $H_N(\CC)$--valued Borel measure $\mu$ such that  
	\begin{equation}\label{moment-measure-matrix}
		S_i=\int_{\RR}x^i d\mu\quad  (i\in \ZZ,\; -2k_1\leq i\leq 2k_2).
	\end{equation}
holds. In \cite{Sim06}, the author gave necessary and sufficient conditions for the solvability of the STHMP of degree $(-2m,2m)$, $m\in \NN$, and also described all solutions in terms of self-adjoint extensions of a certain, not necessarily everywhere defined, linear operator on the finite dimensional Hilbert space of $N$--vector Laurent polynomials. The operator techniques used in \cite{Sim06} are in fact not sensitive to the assumption that $k_1=k_2=m$ and can be verbatim extended to the general degree $(-2k_1,2k_2)$
case, where $k_1,k_2\in \NN$. Moreover, using the same techniques one can also solve the matrix THMP, i.e., the sequence $\beta$ is of degree $(0,2m)$ or even of degree $(2m_1,2m_2)$, where $m,m_1,m_2\in \NN$. Since except solvability we are also interested in the more concrete description of the minimal measures in the scalar STHMP case, where a \textbf{minimal measure} refers to the representing measure with the smallest possible number of atoms, we give a proof of Theorem \ref{strongHamburger-general}  based on the application of Theorem \ref{Hamburger} in Subsection \ref{sub-sec-3-1}.
Then, in Subsection \ref{140222-2346}, we explain the connection with Simonov's work. 
\end{remark}

\subsection{Proof of Theorem \ref{strongHamburger-general} using the solution of the THMP}
\label{sub-sec-3-1}
	First we prove the implication $\ref{pt1-2609-1306} \Rightarrow \ref{rankcard-29-09-1325}$.
 	If $\beta$ is nonsingular, then we have $M(-k_1,k_2)\succ 0$, which is \ref{1712-2257}.
	Else $\beta$ is singular and $M(-k_1,k_2)\not\succ 0$ holds.
	Since $\beta$ admits a measure, it can be extended with $$\beta_{-2k_1-2},\beta_{-2k_1-1},\beta_{2k_2+1},\beta_{2k_1+2}\in \RR$$
	to a sequence $\beta^{(-2k_1-2,2k_2+2)}$ which admits a measure. 
	By	\ref{pt1-prg-2809-1956} and \ref{pt2-nrg-2809-1956} of Proposition \ref{prg-2809-1955}, \ref{pt2-0310-0958} holds.

	Next we prove the implication $\ref{pt2-2609-1309}\Leftarrow \ref{rankcard-29-09-1325}$.
	We separate two cases:\\

\noindent\textbf{Case 1.} $M(-k_1,k_2)\succeq 0$ and  $M(-k_1,k_2)\not\succ 0$: 
			Since \ref{pt2-0310-0958} holds, there exist by Proposition \ref{prg-2809-1955}.\ref{pt1-prg-2809-1956} unique $\beta_{2k_2+1},\beta_{2k_1+2}\in \RR$ such that
			$M(-k_1,k_2+1)\succeq 0$ and $\Rank M(-k_1,k_2)=\Rank M(-k_1,k_2+1)$. Inductively, for every $m\in\NN$ there is a unique extension of $\beta^{(-2k_1,2k_2)}$
			to $\beta^{(-2k_1,2(k_2+m))}$, such that $M(-k_1,k_2+m)\succeq 0$ and $\Rank M(-k_1,k_2)=\Rank M(-k_1,k_2+m)$.
			Write $r=\Rank \beta$ and let $m\in \NN$ be such that $k_2+m\geq r$. Let $p(x)=x^r-\sum_{i=0}^{r-1}\varphi_i x^i$ be the generating polynomial of $\beta^{(-2k_1,2k_2)}$.
			By Theorem \ref{Hamburger}, there exists a unique measure $\mu=\sum_{\ell=1}^r \rho_\ell \delta_{x_\ell}$ for $\beta^{(0,2(k_2+m))}$, where $x_1,\ldots,x_r\in \RR$ are zeroes of $p$
			and $\rho_1,\ldots,\rho_r$ are the corresponding densities. First note by Proposition \ref{prg-nrg-rg-2809-1419} that $\varphi_0\neq 0$ and hence all atoms $x_\ell$ are nonzero.
			We will prove that this is also the representing measure for $\beta^{(-2k_1,0)}$.
			Let us assume that $\mu$ represents $\beta_{j+1},\beta_{j+2},\ldots,\beta_{j+r}$ for some $-2k_1\leq j\leq -1$ and prove that it also represents $\beta_{j}$. Note that for $j=-1$ the
			assumption that $\mu$ represents $\beta_0,\beta_1,\ldots, \beta_{r-1}$ holds and the validity for $j<-1$ will hold by induction. We have:
			\begin{align*}
				\sum_{\ell=1}^{r}\rho_\ell x_\ell^j
					&=\sum_{\ell=1}^{r}\rho_\ell \left(\frac{1}{\varphi_0}x_\ell^{r+j}-\sum_{i=1}^{r-1}\frac{\varphi_i}{\varphi_0} x_\ell^{i+j}\right)
					=\sum_{\ell=1}^{r}\rho_\ell \left(\psi_{r-1}x_\ell^{r+j}+\sum_{i=1}^{r-1}\psi_{i-1} x_\ell^{i+j}\right)\\
					&=\sum_{i=1}^r\psi_{i-1}\left(\sum_{\ell=1}^{r}\rho_\ell x_\ell^{i+j} \right)
					=\sum_{i=1}^r\psi_{i-1}\beta_{i+j}=\beta_{j}.
			\end{align*}
			where the first equality follows by expressing $x_\ell^{j}$ from $\frac{x_\ell^j}{\varphi_0}\cdot p(x_\ell)$ which is equal to 0, 
			the second by Proposition \ref{prg-nrg-rg-2809-1419}.\ref{pt4-prg-nrg-2809-1424}, 
			the forth by the hypothesis that $\mu$ represents $\beta_{j+1},\ldots,\beta_{j+r}$,
			and the last by \ref{pt0-prg-0110-1106} and \ref{pt2-nrg-2809-1956} of Proposition \ref{prg-2809-1955}.
			Hence $\mu$ represents $\beta_j$ and by induction also $\beta^{(-2k_1,0)}$.\\

\noindent\textbf{Case 2.} $M(-k_1,k_2)\succ 0$: By Proposition \ref{140222-1130} there exist $\beta_{2k_2+1}, \beta_{2k_2+2}\in \RR$ such that $\widetilde \beta=(\beta,\beta_{2k_2+1},\beta_{2k_2+2})$
	is singular and rg. By Proposition \ref{prg-2809-1955}.\ref{pt1-prg-2809-1956},\ref{pt2-nrg-2809-1956}, $\widetilde \beta$ satisfies	
		$$\Rank M(-k_1,k_2)=\Rank \left(M(-k_1,k_2+1)(\widetilde\beta)\right)=\Rank \left(M(-k_1+1,k_2+1)(\widetilde\beta)\right).$$
	Now we use Case 1 for $\widetilde\beta$ to establish \ref{rankcard-29-09-1325}.\\

	%This proves the implication $\ref{pt2-2609-1309}\Leftarrow \ref{rankcard-29-09-1325}$. 
	The implication $\ref{pt1-2609-1306}\Leftarrow \ref{pt2-2609-1309}$ is trivial.
	The equivalence $\ref{rg-21-12-1158}\Leftrightarrow\ref{rankcard-29-09-1325}$ follows from Theorem \ref{Hamburger} used for $\beta^{(-2k_1,2k_2)}$
	and its reversed sequence $(\beta^{(-2k_1,2k_2)})^{(\rev)}=(\beta_{2k_2},\beta_{2k_2-1},\ldots,\beta_{-2k_1+1},\beta_{2k_1})$ as $\beta$ to obtain
	the equivalences:
	\begin{itemize}
	\item $\beta^{(-2k_1,2k_2)}$  is prg if and only if $M(-k_1,k_2)\succ 0$ or
		$[M(-k_1,k_2)\succeq 0$ and $\Rank M(-k_1,k_2)=\Rank M(-k_1,k_2-1)]$.
	\item $(\beta^{(-2k_1,2k_2)})^{(\rev)}$ is prg if and only if 
		$\beta^{(-2k_1,2k_2)}$ is nrg if and only if it holds that 
		$M(-k_1,k_2)\succ 0$ or $[M(-k_1,k_2)\succeq 0$ and the equality $\Rank M(-k_1,k_2)=\Rank M(-k_1+1,k_2)]$ is true.
	\end{itemize}
	Using both equivalences gives the equivalence $\ref{rg-21-12-1158}\Leftrightarrow\ref{rankcard-29-09-1325}$. 

	The moreover part can be read out of the proof of the implication $\ref{pt2-2609-1309}\Leftarrow \ref{rankcard-29-09-1325}$. 
	In case $\beta$ is a singular sequence, Case 1 applies, while if $\beta$ is not singular, then Case 2 applies. In Case 1 the constructed representing measure
	is precisely the one stated in \ref{140222-1407}, while in Case 2 precisely singular, rg extensions $\widetilde\beta=(\beta,\beta_{2k_2+1},\beta_{2k_2+2})$ 
	have $(k_1+k_2+1)$--atomic representing measures. By Proposition \ref{140222-1130} these are precisely the ones stated in \ref{140222-1408}.\qed

\subsection{Proof of $\ref{pt1-2609-1306}\Leftrightarrow\ref{pt2-2609-1309}\Leftrightarrow \ref{rankcard-29-09-1325}$ of Theorem \ref{strongHamburger-general} using the operator approach from \cite{Sim06}}
\label{140222-2346}

Let 
	$$\CC^N[x^{-1},x]_{k_1,k_2}=\Span\left\{u x^i\colon u\in \CC^N, i=-k_1,-k_1+1,\ldots,k_2\right\}$$
be a linear space of $N$--vector Laurent polynomials of degree at most $k_1$ in $x^{-1}$ and $k_2$ in $x$.
Let
	$\{S_i\}_{i=-2k_1}^{2k_2}$ be a sequence of hermitian $N\times N$ complex matrices, which is \textbf{positive}, i.e., 
$\sum_{i,j=-k_1}^{k_2} v_j^\ast S_{i+j}v_i\geq 0$ for every sequence $\{v_i\}_{i=-k_1}^{k_2}$ where  $v_i\in \CC^N$.
For a positive sequence $\{S_i\}_{i=-2k_1}^{2k_2}$, the Hermitian form 
	$$\langle u_1 x^i, u_2 x^j   \rangle=u_2^\ast S_{i+j} u_2$$
on $\CC^N[x^{-1},x]_{k_1,k_2}$ is a semi-inner product. Quotienting out the vector subspace 
	$$\mathcal N=\left\{p\in\CC^N[x^{-1},x]_{k_1,k_2}\colon \langle p,p   \rangle=0 \right\}$$
gives a finite dimensional Hilbert space $\cH$. We denote by $[p]:=p+\mathcal N\in \cH$ the equivalence class of $p\in \CC^N[x^{-1},x]_{k_1,k_2}$.

We call the sequence $\{S_i\}_{i=-2k_1}^{2k_2}$:
\begin{itemize}
\item \textbf{matricially positively recursively generated (mat--prg)} if for any sequence $\{v_i\}_{i=-k_1}^{k_2-1}$ with  $v_i\in \CC^N$ the following holds:
	$$\sum_{i,j=-k_1}^{k_2-1} v_j^\ast S_{i+j}v_i\geq 0\quad \text{implies that}\quad \sum_{i,j=-k_1}^{k_2-1} v_j^\ast S_{i+j+2}v_i\geq 0.$$
\item \textbf{matricially negatively recursively generated (mat--nrg)} if for any sequence $\{v_i\}_{i=-k_1}^{k_2-1}$ with $v_i\in \CC^N$ the following holds:
	$$\sum_{i,j=-k_1}^{k_2-1} v_j^\ast S_{i+j+2}v_i\geq 0\quad \text{implies that}\quad \sum_{i,j=-k_1}^{k_2-1} v_j^\ast S_{i+j}v_i\geq 0.$$
\item \textbf{matricially recursively generated (mat--rg)} if it is mat--prg and mat--nrg.
\end{itemize}

If the sequence $\{S_i\}_{i=-2k_1}^{2k_2}$ is mat--prg, the multiplication operator $A([p]):=[xp]$ on $\cH$ with domain 
	$$\dom A:=\Span\left\{[u x^i]\colon u\in \CC^N, i=-k_1,-k_1+1,\ldots,k_2-1\right\}$$
is well-defined. If moreover $\{S_i\}_{i=-2k_1}^{2k_2}$ is mat--nrg, it follows that $\ker A=\{0\}$.

Solution of the moment problem \eqref{moment-measure-matrix} from \cite{Sim06} is the following.

\begin{theorem}\label{140222-2254}\cite[Theorems 3.3 and 3.4, Corollary 3.4.1]{Sim06}
\begin{enumerate}[(1)]
	\item The moment problem \eqref{moment-measure-matrix} is solvable if and only if $\{S_i\}_{i=-2k_1}^{2k_2}$ is positive and matricially recursively generated.
	\item\label{150222-0851} There exists a one-to-one correspondence between the set of all solutions $\mu$ of \eqref{moment-measure-matrix} and the set of all
		equivalence classes for the relation of unitary equivalence of self-adjoint extensions
		$\widetilde A$ of $A$ on some larger Hilbert space $\widetilde \cH\supseteq \cH$, satisfying $\ker \widetilde A=\{0\}$ and
			$$\widetilde \cH=\overline{\Span}\left\{[u],(\widetilde A-\lambda)^{-1}[v]\colon u,v\in \CC^N,\lambda\in \rho(\widetilde A)\right\},$$
		where $\rho(\widetilde A):=\left\{\lambda \in \CC\mid \ker(\widetilde A-\lambda)=0, \Ran(\widetilde A-\lambda)=\widetilde \cH\right\}$ is the resolvent set of $\widetilde A$.
		The correspondence is given by 
			\begin{equation}\label{170222-1302}
				\langle\mu(t)u,v\rangle_{\cH}=\langle E_{\widetilde A}(t)[u],[v]\rangle_{\widetilde \cH},\quad u,v\in \CC^n,
			\end{equation}
		where $E_{\widetilde A}$ is the spectral measure of $\widetilde A$.
	\item\label{150222-0859} The  moment problem \eqref{moment-measure-matrix} has a unique solution if and only if $A$ is self-adjoint.
\end{enumerate}
\end{theorem}

Using Theorem \ref{140222-2254} the equivalence $\ref{pt1-2609-1306}\Leftrightarrow \ref{rankcard-29-09-1325}$ of Theorem \ref{strongHamburger-general}
easily follows by noticing that being positive and mat--rg for $N=1$ is equivalent to satisfying $\ref{rankcard-29-09-1325}$ of Theorem \ref{strongHamburger-general}.

To prove the equivalence $\ref{pt1-2609-1306}\Leftrightarrow\ref{pt2-2609-1309}$ of Theorem \ref{strongHamburger-general} we have to argue in the following way: $A$ is a symmetric operator on the finite dimensional Hilbert space.
If $\dom(A)=\cH$, then $A$ is self-adjoint and by Theorem \ref{140222-2254} its spectral measure, which is supported on the set of eigenvalues of $A$, gives the unique $(\Rank A)$--atomic
representing measure $\mu$ for $\beta$ by the correspondence \eqref{170222-1302}. Since 
\begin{align*}
	\dom(A)=\cH 
	&\Leftrightarrow [x^{k_2}]=\left[\sum_{i=-k_1}^{k_2-1} \alpha_i x^i\right]\quad \text{for some }\alpha_i\in \CC\\
	&\Leftrightarrow \Rank M(-k_1,k_2)=\Rank M(-k_1,k_2-1),
\end{align*} 
this measure is also $(\Rank \beta)$--atomic. 
Otherwise $\dom(A)\subset\cH$ is a linear subspace of codimension 1 in $\cH$ and $A$ can be extended to a self-adjoint invertible operator $\widetilde A$ on $\cH$. 
By Theorem \ref{140222-2254}, its spectral measure, which is 
$\dim \cH=(k_1+k_2+1)$--atomic, gives a $(k_1+k_2+1)$--atomic
representing measure $\mu$ for $\beta$ by the correspondence \eqref{170222-1302}.

\begin{remark}
\begin{enumerate}
\item The moreover part in Theorem \ref{strongHamburger-general} does not directly follow from Theorem \ref{140222-2254}
since one would need to observe more carefully the minimal-rank self-adjoint extensions $\widetilde A$ of $A$ from Theorem \ref{140222-2254}.\ref{150222-0851}
to describe precisely their spectral measures (or equivalently because of finite-dimensionality eigenpairs) in terms of the sequence $\beta$.
\item Using the same technique as above one can give an alternative solution of the matrix THMP (see \cite{AT00,And70,BW11,CH98,Dym89,Ers68}). Replacing $N$--vector Laurent polynomials with $N$--vector polynomials 
		$$\CC^N[x]_{k_1,k_2}:=\Span\left\{ux^i\colon u\in \CC^N, i=k_1,\ldots,k_2\right\},$$
	following the proof of Theorem \ref{140222-2254} in \cite{Sim06} one obtains the fact, that the sequence $\{S_i\}_{i=2k_1}^{2k_2}$ of hermitian $N\times N$ complex matrices admits a $H_N(\CC)$--valued
	Borel measure such that $S_i=\int_{\RR} x^i d\mu$ for each $i$ if and only if $\{S_i\}_{i=2k_1}^{2k_2}$ is positive and mat--prg, while all solutions are  precisely those described in Theorem \ref{140222-2254}.\ref{150222-0851} 
	only that the condition $\ker \widetilde A=\{0\}$ is dropped. (This condition is needed only for the equality $S_{-2k_1}=\int x^{-2k_1}d\mu$ in the STHMP case.) 
	The uniqueness part remains the same as in Theorem \ref{140222-2254}.\ref{150222-0859}. 
\end{enumerate}
\end{remark}

\subsection{The TMP with variety $xy=1$}

As a corollary of Theorem \ref{strongHamburger-general} 
we obtain a new proof of the TMP of degree $2k$ with variety $xy=1$, solved in \cite{CF05}. 
Moreover, our approach shows that in case the representing measure exists, there is always a $(\Rank M(k))$--atomic one.

Let $M(k)$ be a moment matrix associated with a bivariate sequence $\beta^{(2k)}$. We write $(M(k))_{S_1,S_2}$ for 
the restriction of $M(k)$ to rows and columns indexed by the sets $S_1$ and $S_2$, respectively. We also write
$(M(k))_{S}:=(M(k))_{S,S}$ and 
	$\cB:=\{Y^k,\ldots,Y,1,X,\ldots,X^k\}$.

\begin{corollary}\label{posledica-12:51}
	For $k\in\NN$, let $\beta^{(2k)}=(\beta_{0,0},\beta_{1,0},\beta_{0,1},\ldots,\beta_{1,2k-1},\beta_{0,2k})$ 
	be a 2--dimensional sequence of degree $2k$, such that $\beta_{0,0}>0,$ with the associated moment matrix 
	$M(k)$. %which satisfies the relation $XY=1$. 
	Then there exists a representing measure for $\beta^{(2k)}$ supported on $K:=\{(x,y)\in \RR^2\colon xy=1\}$ if and only if the following statements hold:
	\begin{enumerate}[(1)]
		\item\label{pt1-31-12-15:25} One of the following holds:
		\begin{enumerate}	
			\item $k\geq 2$ and $XY=1$ is a column relation.
			\item $k=1$ and $\beta_{1,1}=\beta_{0,0}$.
		\end{enumerate}
		\item\label{pt3-2809-1934} $M(k)$ is positive semidefinite, recursively generated and if $\Rank (M(k))_{\cB}=2k$, then 
					$$\Rank (M(k))_{\cB\setminus \{X^k\}}=
					\Rank (M(k))_{\cB\setminus \{Y^k\}}=2k.$$
	\end{enumerate}

Moreover, let $r=\Rank M(k)$ and $\beta$ admits a $K$--representing measure. Let
$$\widetilde \beta:=(\beta_{0,2k},\beta_{0,2k-1},\ldots,\beta_{0,1},\beta_{0,0},\beta_{1,0},\ldots,\beta_{2k,0}).$$ Then:
\begin{enumerate}[(i)]
	\item\label{140222-1407} If $r\leq 2k$, 
		then the representing measure is unique and of the form
		$\mu=\sum_{i=1}^{r}\rho_i\delta_{(x_i,x_i^{-1})},$ where $x_1,\ldots,x_r$ are the roots of the 
		generating polynomial of $\widetilde \beta$ and $\rho_1,\ldots,\rho_r>0$ the corresponding densities.
	\item\label{140222-1408} If $r=2k+1$, then there are infinitely many $(2k+1)$--atomic representing measures for $\beta$.
		Denoting $A_{\widetilde \beta}=\left(\begin{array}{cccc} \mbf{v_0} & \mbf{v_1} & \cdots & \mbf{v_{2k}} \end{array}\right)$, 
		they are obtained by the following procedure:
		\begin{itemize}
	 		\item Choose any $\beta_{2k+1,0}\in \RR$, which is not equal to $v^TC^{-1}v$ if $C$ is invertible, where 	
					$C=\left(\begin{array}{ccc} \mbf{v_1}(0\mathbin{:}k_1+k_2-1) & \cdots & \mbf{v_k}(0\mathbin{:}k_1+k_2-1) \end{array}\right)$
				and $v=\left(\begin{array}{cccc} \beta_{1,0} & \cdots & \beta_{2k,0}\end{array}\right)^T$.
			\item Define	$\beta_{2k+2,0}=w^T(A_{\widetilde \beta})^{-1}w,$
				where $w=\left(\begin{array}{cccc} \beta_{1,0}&\cdots&\beta_{2k,0}&\beta_{2k+1,0}\end{array}\right)^T$.
			\item Use \ref{140222-1407} for $\widehat\beta:=(\widetilde \beta,\beta_{2k+1,0},\beta_{2k+2,0})$.	
		\end{itemize}
\end{enumerate}
\end{corollary}

\begin{proof}
	For $m\in \{-2k,-2k+1,\ldots,2k\}$ we define the numbers $\beta_m$ by the following rule
		$$\beta_{m}
			=\left\{\begin{array}{cc}   
				\beta_{m,0},& m\geq 0,\\
				\beta_{0,-m},& m<0.
			\end{array}\right.$$

	\noindent \textbf{Claim.} 
		Let $t\in \NN$. 
		The atoms $(x_1,x_1^{-1}),\ldots (x_t,x_t^{-1})$ with densities $\rho_1,\ldots,\rho_t$
		are the $(xy-1)$-representing measure for $\beta^{(2k)}=(\beta_{i,j})_{i,j\in \ZZ^2_+,i+j\leq 2k}$
		if and only if
		the atoms $x_1,\ldots,x_t$ with densities $\rho_1,\ldots,\rho_t$
		are the $(\RR\setminus\{0\})$-representing measure for the 1--dimensional sequence $\beta:=(\beta_{-2k},\ldots,\beta_{-1},\beta_0,\beta_1,\ldots,\beta_{2k})$.\\

	The only if part follows from the following calculation:
	\begin{equation*}
	  \beta_{i,j}
			=\beta_{i-1,j-1}=\ldots
			=\left\{\begin{array}{cc}   
				\beta_{i-j,0},& i\geq j,\\
				\beta_{0,j-i},& i<j.
				\end{array}\right.
			=\beta_{i-j}=\sum_{\ell=1}^t \rho_\ell x_\ell^{i-j}=\sum_{\ell=1}^t \rho_\ell x_\ell^{i} (x_{\ell}^{-1})^j,
	\end{equation*}
	where $i,j\in \ZZ^2_+$ such that $i+j\leq 2k$.
	
	The if part follows from the following calculation:
		$$\beta_{m}
			=\left\{\begin{array}{cc}   
				\beta_{m,0},& m\geq 0,\\
				\beta_{0,-m},& m<0.
			\end{array}\right.
			=\left\{\begin{array}{rl}   
				\sum_{\ell=1}^t \lambda_\ell x_\ell^{m},& m\geq 0,\\
				\sum_{\ell=1}^t \lambda_\ell (x_\ell^{-1})^{-m},& m<0.
			\end{array}\right.
			=\sum_{\ell=1}^t \lambda_\ell x_\ell^{m},$$
	where $m=-2k,-2k+1,\ldots,2k$.\\

	Using Claim, a theorem of Bayer and Teichmann \cite{BT06}, implying that if $\beta^{(2k)}$ has a $K$--representing measure, 
	then it has a finitely atomic $K$--representing measure, and Theorem \ref{strongHamburger-general}, 
	there exists a representing measure for $\beta^{(2k)}$ supported on $K$ if and only if \ref{pt1-31-12-15:25} and \ref{pt2-31-12-15:30} are true,
	where
	\begin{enumerate}[label=(\Alph*)]
	\item\label{pt2-31-12-15:30} $M(k)$ is psd, rg and one of the following conditions holds:
	\begin{enumerate}[(a)]	
			\item\label{pta-31-12-15:39} $(M(k))_{\cB}\succ 0$.
			\item\label{case2-21-12-9:05} $\Rank (M(k))_{\cB}=
					\Rank (M(k))_{\cB\setminus \{X^k\}}=
					\Rank (M(k))_{\cB\setminus \{Y^k\}}$.
	\end{enumerate}
	\end{enumerate}
	It remains to prove the equivalence \ref{pt2-31-12-15:30} $\Leftrightarrow$ \ref{pt3-2809-1934}. 
	The nontrivial implication is \ref{pt2-31-12-15:30} $\Leftarrow$ \ref{pt3-2809-1934}.
	If $\Rank (M(k))_{\cB}=2k+1$, then \ref{pta-31-12-15:39} follows form the fact that $M(k)$ is psd. 
	If $\Rank (M(k))_{\cB}=2k$, then we are in case \ref{case2-21-12-9:05}.
	It remains to prove that in case $\Rank (M(k))_{\cB}<2k$, $(M(k))_{\cB}$ being psd and rg implies \ref{case2-21-12-9:05}. 
	By symmetry it suffices to prove that $\Rank (M(k))_{\cB}=\Rank (M(k))_{\cB\setminus\{X^k\}}$. Let us assume on contrary that   
	$\Rank (M(k))_{\cB}>\Rank (M(k))_{\cB\setminus\{X^k\}}$. This means that 
		$$\Rank (M(k))_{\cB\setminus \{X^k\}}\leq 2k-2.$$
	Since $(M(k))_{\cB}$ is a Hankel matrix in the order $Y^k,\ldots,Y,1,X,\ldots,X^k$ of rows and columns, it follows that 
		$$X^{k-2}\in \Span\{Y^{k},\ldots,Y,1,X,\ldots,X^{k-3}\}.$$	
	Proposition \ref{prg-2809-1955} implies that
		$$X^{k-1}\in \Span\{Y^{k-1},\ldots,Y,1,X,\ldots,X^{k-2}\}$$
	or equivalently
	\begin{equation}\label{relation-31-12-16:46}
		\displaystyle X^{k-1}=\sum_{i=k-1}^1\alpha_i Y^i+\sum_{j=0}^{k-2}\beta_jX^j\quad\text{for some }\alpha_i,\beta_j\in\RR.
	\end{equation}
	Since $M(k)$ is rg, multiplying \ref{relation-31-12-16:46} with $X$ and using $XY=1$, implies that
		$$X^{k}\in \Span\{Y^{k-2},\ldots,Y,1,X,\ldots,X^{k-1}\},$$ 
	which is a contradiction with 
		$\Rank (M(k))_{\cB}>\Rank (M(k))_{\cB\setminus\{X^k\}}$.
	This proves \ref{pt2-31-12-15:30} $\Leftarrow$ \ref{pt3-2809-1934}.

	The moreover part of the corollary follows from the moreover part of Theorem \ref{strongHamburger-general} by also noticing that 
	$\beta=\widetilde\beta$ and
		$\Rank \widetilde \beta=\Rank (M(k))_\cB=\Rank M(k).$
	This concludes the proof of the corollary.
\end{proof}

\begin{remark}
\cite[Proposition 2.14]{CF05}
states that in case $\Rank M(k)=2k+1$ there exists a $(\Rank M(k))$ or $(\Rank M(k)+1)$--atomic measure, depending on the choice of the moments $\beta_{2k+1,0}$ and 
$\beta_{0,2k+1}$ in the extension $M(k+1)$ (denoted by $p$ and $q$ in the proof of \cite[Proposition 2.14]{CF05}). By Corollary \ref{posledica-12:51}, $p$ and $q$ giving a 
$(\Rank M(k))$--atomic measure, exist. Note also that this is not in contradiction with \cite[Example 5.2]{CF05} which only demonstrates the role of the choices of $\beta_{2k+1,0}$ and 
$\beta_{0,2k+1}$ on the rank of the extension of $M(k)$ to the moment matrix $M(k+1)$.
\end{remark}

%%%%%%%%%%%%%%%%%%%%%%%%%%%%%%%%%%%%%%%%%%%%%%%%%%%%%%%%%%%%%%%%%%%%
%%%%%%%%%%%%%%%%%%%%%%%%%%%%%%%%%%%%%%%%%%%%%%%%%%%%%%%%%%%%%%%%%%%%
%%%%%%%%%%%%%%%%%%%%%%%%%%%%%%%%%%%%%%%%%%%%%%%%%%%%%%%%%%%%%%%%%%%%

\section{The STHMP with a gap $\beta_{-2k_1+1}$ or $\beta_{2k_2-1}$ and the TMP with variety $x^2y=1$}
\label{S4}

In this section we first solve the STHMP of degree $(-2k_1,2k_2)$ with a
missing moment
$\beta_{-2k+1}$ or 
$\beta_{2k_2-1}$ %(see Theorem \ref{strong-trunc-Hamb-without-(-2k+1)}).
%and $(\beta_{-2k+1},\beta_{-2k+2})$
%	(see Theorem \ref{strong-trunc-Hamb-without-(-2k+1)-(-2k+2)}). 
and then as a corollary %of Theorem \ref{strong-trunc-Hamb-without-(-2k+1)} 
obtain the solution to the TMP for the curve $x^2y=1$. %(see Corollary \ref{X2Y=-general}).
%while as a corollary of Theorem \ref{strong-trunc-Hamb-without-(-2k+1)-(-2k+2)} 
%we get the solution to the TMP for the curve $x^2y^3=1$ and an additional moment $\beta_{\frac{5}{3},0}$ given
%(see Corollary \ref{x^2y^3-gen}).

%\subsection{Strong truncated Hamburger moment problem of degree $(-2k_1,2k_2)$ with gaps $(\beta_{-2k_1+1})$ or $(\beta_{2k_2-1})$}

%The following theorem is a solution of the truncated Hambruger moment problem of degree $2k$ with a gap 
%$(\beta_1)$.

A \textbf{partial matrix} $A=(a_{ij})_{i,j=1}^n$ is a matrix of real numbers $a_{ij}\in \RR$, where some of the entries are not specified. 
A symmetric matrix $A=(a_{ij})_{i,j=1}^n$ is 
\textbf{partially positive semidefinite (ppsd)} 
%(resp.\ \textbf{partially positive definite (ppd)}) 
if the following two conditions hold:
\begin{enumerate} 
  \item $a_{ij}$ is specified if and only if $a_{ji}$ is specified and $a_{ij}=a_{ji}$.
  \item All fully specified principal minors of $A$ are psd.% (resp.\ pd). 
\end{enumerate}

Let 
	\begin{equation}\label{matrixM}
		M=\left[ \begin{array}{cc} A & B \\ C & D \end{array}\right]\in M_{n+m}
	\end{equation}
be a real matrix where $A\in M_n$, $B\in M_{n,m}$, $C\in M_{m,n}$  and $D\in M_{m}$.
The \textbf{generalized Schur complement} \cite{Zha05} of $A$ (resp.\ $D$) in $M$ is defined by
	$$M/A=D-CA^\dagger B\quad(\text{resp.}\; M/D=A-BD^\dagger C),$$
where $A^\dagger$ (resp.\ $D^\dagger$) stands for the Moore-Penrose inverse of $A$ (resp.\ $D$).

\begin{theorem}\label{strong-trunc-Hamb-without-(-2k+1)}
 Let $k_1,k_2\in \NN$, and 
	$$\beta(x):=(\beta_{-2k_1},x,\beta_{-2k_1+2},\ldots,\beta_{0},\ldots,\beta_{2k_2})$$ 
  be a sequence where each $\beta_i$ is a real number, $\beta_{-2k_1}>0$ and $x$ is a variable. 
  Let 
	$$v:=\left(\begin{array}{ccc}\beta_{-2k_1+2} & \cdots & \beta_{-k_1+k_2-1}\end{array}\right)\quad\text{and}\quad
		u:=\left(\begin{array}{ccc}\beta_{-2k_1+2} & \cdots & \beta_{-k_1+k_2}\end{array}\right)$$
  vectors, and
	$$\widetilde{A}:=\left(\begin{array}{cc} \beta_{-2k_1} & v \\ v^T & M(-k_1+2,k_2-1) \end{array}\right)
		\;\;\text{and}\;\;
			\widehat{A}:=\left(\begin{array}{cc} \beta_{-2k_1} & u \\ u^T & M(-k_1+2,k_2) \end{array}\right)$$
  matrices.
  Then the following statements are equivalent: 
\begin{enumerate}[(1)]		
	\item\label{pt1-v2509-14:19} There exists $x_0\in \RR$ and a representing measure for $\beta(x_0)$ supported on $K=\RR\setminus \{0\}$.
	\item\label{pt4-v2112-16:49v2} There exists $x_0\in \RR$ such that $\beta(x_0)$ is recursively generated.
	\item\label{pt5-2312-01:53} There exists $x_0\in \RR$ such that $\beta(x_0)$ is singular and recursively generated.
	\item\label{pt2-2312-01:59} There exists $x_0\in \RR$ and a $(\Rank M(-k_1+1,k_2))$--atomic representing measure for $\beta(x_0)$ supported on $K=\RR\setminus \{0\}$.
	\item\label{pt3-v2509-14:21v2} $A_{\beta(x)}$ is partially positive semidefinite and one of the following conditions is true: 
		\begin{enumerate}[(a)]
			\item\label{pt3a-v2509-14:21} $M(-k_1+1,k_2)\succ 0$ and $\widetilde{A}\succ 0$.
			\item\label{pt3b-v2509-14:21} 
				$\Rank M(-k_1+1,k_2-1)=\Rank M(-k_1+1,k_2)=\Rank M(-k_1+2,k_2)=\Rank \widehat A.$
		\end{enumerate}
\end{enumerate}

Moreover, assume that there exists $x_0\in \RR$ such that \ref{pt2-2312-01:59}  holds. Let 
$$s:=M(-k_1+1,k_2)\big/M(-k_1+2,k_2),\quad t:=\widehat A\big/M(-k_1+2,k_2)$$ and
$w=\left(\begin{array}{ccc}\beta_{-2k_1+3}&\cdots&\beta_{-k_1+k_2+1}\end{array}\right)$. Then:
\begin{enumerate}[(i)]
	\item\label{150222-1028} If $s=t=0$, then $x_0:=u{(M(-k_1+2,k_2))}^{\dagger} w^T$.
		%the representing measure for $\beta(x_0)$ is unique and its support consists of the roots of the generating polynomial of $\beta(x_0)$.
	\item\label{150222-1029} Else $s>0$, $t>0$ and there are two choices $x_{0,\pm}$ for $x_0$, i.e., 
		$$x_{0,\pm}=u(M(-k_1+2,k_2))^{\dagger}w^T\pm \sqrt{s\cdot t}.$$
\end{enumerate}
Once $x_0$ is fixed, the representing measure for $\beta(x_{0})$ is unique and its support consists of the roots of the generating polynomial of $\beta(x_0)$.
\end{theorem}

\begin{proof}[Proof of Theorem \ref{strong-trunc-Hamb-without-(-2k+1)}]
The equivalence $\ref{pt1-v2509-14:19}\Leftrightarrow\ref{pt4-v2112-16:49v2}$ follows from Theorem \ref{strongHamburger-general}.

Now we prove the implication $\ref{pt4-v2112-16:49v2}\Rightarrow\ref{pt3-v2509-14:21v2}$.
Since $\beta(x_0)$ is rg, the sequence $\beta(x_0)$ and the reversed sequence $\beta(x_0)^{(\rev)}:=(\beta_{2k_2},\beta_{2k_2-1},\ldots,\beta_{-2k_1+2},x_0,\beta_{-2k_1})$ 
are both prg. Regarding $\beta(x_0)$ and $\beta(x_0)^{(\rev)}$ as degree $(0,2(k_1+k_2))$ sequences,
the equivalence $\ref{pt1-130222-1851}\Leftrightarrow \ref{pt3-130222-1851}$ of Theorem \ref{Hamburger} implies that they both admit representing measures on $\RR$.
Using \cite[Theorem 4.1]{Zal+} for $\beta(x_0)$ and \cite[Theorem 3.1]{Zal+} for $\beta(x_0)^{(\rev)}$, \ref{pt3-v2509-14:21v2} holds.

The implication $\ref{pt3-v2509-14:21v2}\Rightarrow\ref{pt5-2312-01:53}$ follows from \cite[Theorem 4.1]{Zal+}. Indeed, the equivalence $(\text{ii})\Leftrightarrow (\text{iii})$ of \cite[Theorem 4.1]{Zal+} implies that there exists $x_0\in \RR$ such that $\beta(x_0)$, regarded as a $(0,2(k_1+k_2))$--degree sequence, admits a $(\Rank M(-k_1+1,k_2))$--atomic $\RR$--representing measure. 
So $\beta(x_0)$ is a singular sequence. By Theorem \ref{Hamburger}, $\beta(x_0)$ is prg and $\Rank A_{\beta(x_0)}=\Rank M(-k_1+1,k_2)$. It remains to prove that $\beta(x_0)$ is nrg. 
Let $p(x):=x^r-\sum_{j=0}^{r-1}\varphi_i x^i$ be the generating polynomial of $\beta(x_0)$. If $\varphi_0=0$, then the first column of $A_{\beta(x_0)}$ is not in the span of its other columns. But this is in contradiction with  $\Rank A_{\beta(x_0)}=\Rank M(-k_1+1,k_2)$. Hence, $\varphi_0\neq 0$ and by Proposition \ref{prg-nrg-rg-2809-1419}, $\beta(x_0)$ is rg.
 
The implication $\ref{pt5-2312-01:53}\Rightarrow\ref{pt4-v2112-16:49v2}$
is trivial. So far we established the equivalences $\ref{pt1-v2509-14:19}\Leftrightarrow \ref{pt4-v2112-16:49v2}\Leftrightarrow \ref{pt5-2312-01:53} \Leftrightarrow \ref{pt3-v2509-14:21v2}$. 
The implication $\ref{pt2-2312-01:59}\Rightarrow\ref{pt1-v2509-14:19}$ is trivial. It remains to prove the implication $\ref{pt5-2312-01:53}\Rightarrow\ref{pt2-2312-01:59}$. 
Since $\beta(x_0)$ is singular and rg, 
Proposition \ref{prg-2809-1955}.\ref{pt1-prg-2809-1956} implies that  $\Rank \beta(x_0)=\Rank A_{\beta(x_0)}$,
while Proposition \ref{prg-2809-1955}.\ref{pt2-nrg-2809-1956} implies that $\Rank A_{\beta(x_0)}=\Rank M(-k_1+1,k_2)$. Hence, $\Rank \beta(x_0)=\Rank M(-k_1+1,k_2)$. 
Using Theorem \ref{strongHamburger-general} for $\beta(x_0)$ gives \ref{pt2-2312-01:59}.

For the moreover part about possible choices of $x_0$ see the Claim in the proof of \cite[Theorem 4.1]{Zal+}. The last sentence about the form of the representing measure for $\beta(x_0)$
follows from the moreover part of Theorem \ref{Hamburger}.
\end{proof}

\begin{corollary}\label{strong-trunc-Hamb-without-(2k-1)}
 Let $k_1,k_2\in \NN$, and 
	$$\beta(x):=(\beta_{-2k_1},\ldots,\beta_{2k_2-2},x,\beta_{2k_2})$$ 
  be a sequence where each $\beta_i$ is a real number, $\beta_{-2k_1}>0$ and $x$ is a variable. 
  Let 
%	$$\widehat{\beta}:=(\beta_{-2k_1+2},\ldots,\beta_{2k_2-2}),\quad
%		\widetilde{\beta}:=(\beta_{-2k_1},\ldots,\beta_{2k_2-2}),\quad
%			\overline{\beta}:=(\beta_{-2k_1+2},\ldots,\beta_{2k_2-4}),$$
%	$$\widebreve{\beta}:=(\beta_{-2k_1},\ldots,\beta_{2k_2-4})$$
%  be subsequences of $\beta(x)$,
	$$v:=\left(\begin{array}{ccc}\beta_{-k_1+k_2+1} & \cdots & \beta_{2k_2-2}\end{array}\right)\quad\text{and}\quad
		u:=\left(\begin{array}{ccc}\beta_{-k_1+k_2}& \cdots & \beta_{2k_2-2}\end{array}\right)$$
  vectors, and
	$$\widetilde{A}:=\left(\begin{array}{cc}M(-k_1+1,k_2-2) & v \\ v^T &  \beta_{2k_2}\end{array}\right)
		\quad\text{and}\quad
			\widehat{A}:=\left(\begin{array}{cc}M(-k_1,k_2-2)& u \\ u^T &  \beta_{2k_2}\end{array}\right)$$
  matrices.
  Then the following statements are equivalent: 
\begin{enumerate}[(1)]	
	\item There exists $x_0\in \RR$ and a representing measure for $\beta(x_0)$ supported on $K=\RR\setminus \{0\}$.
	\item There exists $x_0\in \RR$ such that $\beta(x_0)$ is recursively generated.
	\item There exists $x_0\in \RR$ such that $\beta(x_0)$ is singular and recursively generated.
	\item\label{160222-1630}  There exists $x_0\in \RR$ and a $(\Rank M(-k_1,k_2-1))$--atomic representing measure for $\beta(x_0)$ supported on $K=\RR\setminus \{0\}$.
	\item $A_{\beta(x)}$ is partially positive semidefinite and one of the following conditions is true: 
		\begin{enumerate}
			\item $M(-k_1,k_2-1)\succ 0$ and $\widetilde{A}\succ 0$.
%			  \begin{enumerate}
%%				\item $k=2$ and $A_{\widetilde{\beta}}\succ 0$.
%				\item $A_{\widetilde{\beta}}\succ 0$ and $\widetilde{A}\succ 0$.
%			  \end{enumerate}
			\item
				$\Rank M(-k_1+1,k_2-1)=\Rank M(-k_1,k_2-1)=\Rank M(-k_1,k_2-2)=\Rank \widehat A.$
		\end{enumerate}
\end{enumerate}
Moreover, assume that there exists $x_0\in \RR$ such that \ref{160222-1630} holds. Let 
	$$s:=M(-k_1,k_2-1)\big/M(-k_1,k_2-2),\quad t:=\widehat A\big/M(-k_1,k_2-2)$$ 
and
	$w=\left(\begin{array}{ccc}\beta_{-k_1+k_2-1}&\cdots&\beta_{2k_2-3}\end{array}\right)$. 
Then:
\begin{enumerate}[(i)]
	\item\label{150222-1028} If $s=t=0$, then $x_0:=u{(M(-k_1,k_2-2))}^{\dagger}w^T$.
		%the representing measure for $\beta(x_0)$ is unique and its support consists of the roots of the generating polynomial of $\beta(x_0)$.
	\item\label{150222-1029} Else $s>0$, $t>0$ and there are two choices $x_{0,\pm}$ for $x_0$, i.e., 
		$$x_{0,\pm}=u(M(-k_1,k_2-2))^{\dagger}w^T\pm \sqrt{s\cdot t}.$$
\end{enumerate}
Once $x_0$ is fixed, the representing measure for $\beta(x_{0})$ is unique and its support consists of the roots of the generating polynomial of $\beta(x_0)$.
\end{corollary}

\begin{proof}
	Note that $\sum_{j=1}^\ell \rho_j \delta_{x_j}$, where $\rho_j>0$ are densities and $x_j\in \RR\setminus\{0\}$ are atoms, 
	is a $(\RR\setminus \{0\})$--representing measure for $\beta(x)$  if and only if 
	$\sum_{j=1}^\ell \rho_j \delta_{x_j^{-1}}$ is a $(\RR\setminus\{0\})$--representing measure for
		$$\widetilde \beta(x):=(\tilde\beta_{-2k_2},x,\tilde \beta_{-2k_2+2},\ldots,\tilde\beta_0,\ldots,\tilde \beta_{2k_1}),$$
	where $\tilde\beta_{i}=\beta_{-i}$ for each $i$.
	Using Theorem \ref{strong-trunc-Hamb-without-(-2k+1)}, the corollary follows.
\end{proof}

The following corollary is a consequence of Theorem \ref{strong-trunc-Hamb-without-(-2k+1)} and gives the solution of the bivariate TMP for the curve $x^2y=1$. 

\begin{corollary}
	\label{X2Y=-general}
	Let $\displaystyle\beta=(\beta_{i,j})_{i,j\in \ZZ^2_+,i+j\leq 2k}$ 
	be a 2--dimensional real multisequence of degree $2k$. 
	Suppose $\mc M(k)$ is positive semidefinite and recursively generated.
	Let 
	$$u^{(i)}:=(\beta_{0,i},\beta_{1,i})\quad\text{for }i=1,\ldots,2k-1,$$
	\begin{align*}
		\widehat{\beta} 
			&:=(u^{(2k-1)},u^{(2k-2)},\ldots,u^{(1)},\beta_{0,0},\beta_{1,0},\ldots,\beta_{2k-2,0}),\quad
				\widetilde{\beta}:=(\widehat \beta,\beta_{2k-1,0},\beta_{2k,0}),\\
		\overline{\beta}
			&:=(u^{(2k-2)},u^{(2k-3)},\ldots,u^{(1)},\beta_{0,0},\beta_{1,0},\ldots,\beta_{2k-2,0}),\quad
		\widebreve{\beta}:=(\overline \beta,\beta_{2k-1,0},\beta_{2k,0}),
	\end{align*}
	be subsequences of $\beta$,
	$$v:=\left\{\begin{array}{rl}
			\begin{mpmatrix} u^{(2k-1)} & u^{(2k-2)} & \cdots & u^{(\frac{k}{2}+1)} \end{mpmatrix},& \text{if }k\text{ is even},\\
			\begin{mpmatrix} u^{(2k-1)} & u^{(2k-2)} & \cdots & u^{(\lceil \frac{k}{2}\rceil+1)} & \beta_{0,\lceil \frac{k}{2}\rceil} \end{mpmatrix},& \text{if }k\text{ is odd},
		\end{array}\right.$$
%	and
%	$$u:=\left\{\begin{array}{rl}
%			\begin{mpmatrix} u^{(2k-1)} & u^{(2k-2)} & \cdots & u^{(\frac{k}{2}+1)} & \beta_{0,\lfloor \frac{k}{2}\rfloor}\end{mpmatrix},& \text{if }k\text{ is even},\\
%			\begin{mpmatrix} u^{(2k-1)} & u^{(2k-2)} & \cdots & u^{(\lceil \frac{k}{2}\rceil)} \end{mpmatrix},& \text{if }k\text{ is odd},
%		\end{array}\right.$$
 	a vector and
	$$\widetilde{A}:=\left(\begin{array}{cc} \beta_{0,2k} & v \\ v^T & A_{\overline{\beta}} \end{array}\right)$$
%		\quad\text{and}\quad
%			\widehat{A}:=\left(\begin{array}{cc} \beta_{-4k} & u \\ u^T & A_{\widebreve{\beta}} \end{array}\right)$$
  a matrix.
  Then $\beta$ has a representing measure supported on the variety 
	$K:=\{(x,y)\in \RR^2\colon x^2y=1\}$ if and only if the following statements hold:
	\begin{enumerate}[(1)]
		\item One of the following holds: 
			\begin{itemize}
				\item $k\geq 3$ and $X^2Y=1$ is a column relation of $M(k)$. 
				\item $k=2$ and the equalities 
					$\beta_{2,1}=\beta_{0,0}$, $\beta_{3,1}=\beta_{1,0}$ hold.
				\item $k=1$.
			\end{itemize}
		\item\label{point1-2709-1322} One of the following holds: 
			\begin{enumerate}[(a)]
				\item $A_{\widetilde{\beta}}\succ 0$ and $\widetilde{A}\succ 0$.	
				\item\label{point2-23-12-08:44} $A_{\widetilde\beta}\succeq 0$ and 
					$\Rank A_{\widehat\beta}=\Rank A_{\widetilde \beta}=\Rank A_{\breve \beta}=\Rank M(k)$.
			\end{enumerate}
	\end{enumerate}

Moreover, let $r=\Rank M(k)$ and $\beta$ admits a $K$--representing measure. Let
$\gamma(x):=(\beta_{0,2k},x,\widetilde\beta)$ and $\widehat{A}:=\left(\begin{array}{cc} \beta_{0,2k} & u \\ u^T & A_{\breve\beta} \end{array}\right)$,
where 
		$$u:=\left\{\begin{array}{rl}
			\begin{mpmatrix} v & \beta_{0,\frac{k}{2}}\end{mpmatrix},& \text{if }k\text{ is even},\\
			\begin{mpmatrix} v & \beta_{1,\lceil \frac{k}{2}\rceil}\end{mpmatrix},& \text{if }k\text{ is odd},
		\end{array}\right.$$
and
	$$w:=\left\{\begin{array}{rl}
			\begin{mpmatrix} \beta_{1,2k-1} & u^{(2k-2)} & \cdots & u^{(\frac{k}{2})} \end{mpmatrix},& \text{if }k\text{ is even},\\
			\begin{mpmatrix} \beta_{1,2k-1} & u^{(2k-2)} & \cdots & u^{(\lceil \frac{k}{2}\rceil)} & \beta_{0,\lceil \frac{k}{2}\rceil-1} \end{mpmatrix},& \text{if }k\text{ is odd}.
		\end{array}\right.$$
Then:
\begin{enumerate}[(i)]
	\item If $r< 3k$, 
		then the representing measure is unique and of the form
		$\mu=\sum_{i=1}^{r}\rho_i\delta_{(x_i,x_i^{-2})},$ where $x_1,\ldots,x_r$ are the roots of the 
		generating polynomial of $\gamma(x_0)$,
		$x_0=u{(A_{\breve\beta})}^{\dagger} w^T$ and $\rho_1,\ldots,\rho_r>0$ are the corresponding densities.
	\item If $r=3k$, then there are two $(3k)$--atomic representing measures. Let
			$$x_{\pm}=u{(A_{\breve\beta})}^{\dagger} w^T\pm \sqrt{ \left(A_{\widetilde\beta} \big/ A_{\breve\beta}\right)\cdot\left(\widehat A \big/ A_{\breve\beta}\right)}.$$
		Then the two measures are of the form $\mu=\sum_{i=1}^{r}\rho_{i,\pm}\delta_{(x_{i,\pm},x_{i,\pm}^{-2})},$   
		where $x_{1,\pm},\ldots,x_{r,\pm}$ are the roots of the generating polynomial of $\gamma(x_{\pm})$,
		and $\rho_{1,\pm}$,$\ldots$,$\rho_{r,\pm}>0$ are the corresponding densities.
\end{enumerate}
\end{corollary}

\begin{proof}
	For $m\in \{-4k,-4k+2,-4k+3,\ldots,2k\}$ we define the numbers $\widetilde \beta_m$ by the following rule
		\begin{equation*}%\label{15:04-170720}
			\widetilde \beta_m:=
			\left\{\begin{array}{rl}
				\beta_{0,\frac{|m|}{2}},& \text{if }m \text{ is even and } m<0,\\
				\beta_{1,\lceil\frac{|m|}{2}\rceil},& \text{if }m \text{ is odd and } m<0,\\
				\beta_{m,0},& \text{if }m\geq 0.
			\end{array}\right.
		\end{equation*}
	\noindent \textbf{Claim 1.} Every number $\widetilde \beta_m$ is well-defined.\\

	We have to prove that $i+j\leq 2k$, where $i,j$ are indices of $\beta_{i,j}$ used in the definition of 
	$\widetilde \beta_m$.
	We separate three cases according to $m$:
	\begin{itemize}
		\item $m$ is even and $m<0$: $\frac{|m|}{2}\leq \frac{4k}{2}=2k$.
		\item $m$ is odd and $m<0$: $\lceil\frac{|m|}{2}\rceil+1\leq \lceil\frac{4k-3}{2}\rceil+1=2k-1+1=2k.$
		\item $m$ is nonnegative: $m\leq 2k.$\\
	\end{itemize}
	
	\noindent \textbf{Claim 2.} 
		Let $t\in \NN$. 
		The atoms $(x_1,x_1^{-2}),\ldots (x_t,x_t^{-2})$ with densities $\rho_1,\ldots,\rho_t$
		are the $(x^2y-1)$--representing measure for $(\beta_{i,j})_{i,j\in \ZZ^2_+,i+j\leq 2k}$
		if and only if
		the atoms $x_1,\ldots,x_t$ with densities $\rho_1,\ldots,\rho_t$
		are the $(\RR\setminus\{0\})$--representing measure for 
		$\widetilde \beta(x)=(\widetilde \beta_{-4k},x,\widetilde \beta_{-4k+2},\widetilde \beta_{-4k+3}\ldots,\widetilde\beta_{2k})$.\\

	The if part follows from the following calculation:
	\begin{align*}
	\widetilde \beta_{m}
			&=\left\{\begin{array}{rl}
				\beta_{0,\frac{|m|}{2}},& \text{if }m \text{ is even and } m<0,\\
				\beta_{1,\lceil\frac{|m|}{2}\rceil},& \text{if }m \text{ is odd and } m<0,\\
				\beta_{m,0},& \text{if }m\geq 0,
				\end{array}\right.\\
			&=\left\{\begin{array}{rl}
				\sum_{\ell=1}^t \rho_\ell (x_{\ell}^{-2})^{\frac{|m|}{2}},&\text{if }m \text{ is even and } m<0,\\
				\sum_{\ell=1}^t \rho_\ell x_{\ell}(x_{\ell}^{-2})^{\lceil\frac{|m|}{2}\rceil},& 
					 \text{if }m \text{ is odd and } m<0,\\
				\sum_{\ell=1}^t \rho_\ell x_{\ell}^m,& 
					\text{if }m\geq 0,\\
				\end{array}\right.
			=\sum_{\ell=1}^t \rho_\ell x_\ell^{m},
	\end{align*}
	where $m=-4k,-4k+2,-4k+3,\ldots,2k$.

	The only if part follows from the following calculation:
	\begin{align*}
	  \beta_{i,j}
		&= \beta_{i-2,j-1}=\cdots=\left\{
		\begin{array}{rl}
			\beta_{i-2j,0},& \text{if }i-2j\geq 0,\\
			\beta_{i \Mod{2},j-\lfloor \frac{i}{2}\rfloor},& \text{if }i-2j<0,	
		\end{array}
		\right.
		=\widetilde \beta_{i-2j}\\
		&=\sum_{\ell=1}^t \rho_\ell x_\ell^{i-2j}=\sum_{\ell=1}^t \rho_\ell x_\ell^{i}(x_\ell^{-2})^j,
	\end{align*}
	where the first three equalities in the first line follow by $M(k)$ being rg.\\

	Using Claim 2 and a theorem of Bayer and Teichmann \cite{BT06}, implying that if $\beta$
	has a $K$--representing measure, then it has a finitely atomic $K$--representing measure, the statements of the 	
	corollary follows by Theorem \ref{strong-trunc-Hamb-without-(-2k+1)}. 
\end{proof}

%%%%%%%%%%%%%%%%%%%%%%%%%%%%%%%%%%%%%%%%%%%%%%%%%%%%%%%%%%%%%%%%%%%%
%%%%%%%%%%%%%%%%%%%%%%%%%%%%%%%%%%%%%%%%%%%%%%%%%%%%%%%%%%%%%%%%%%%%
%%%%%%%%%%%%%%%%%%%%%%%%%%%%%%%%%%%%%%%%%%%%%%%%%%%%%%%%%%%%%%%%%%%%




\begin{thebibliography}{00}
\bibitem
{AT00}
	V.M.\ Adamyan, I.M.\ Tkachenko, Solution of the Truncated Matrix Hamburger Moment Problem According to M.G. Krein. 
	In: Adamyan V.M. et al. (eds) Operator Theory and Related Topics. Operator Theory: Advances and Applications, vol 118. Birkhäuser, Basel. 
	\url{https://doi.org/10.1007/978-3-0348-8413-6_3}.	

\bibitem
{Akh65}
	N.I.\ Akhiezer, 
		The classical moment problem and some related questions in analysis,
			Hafner Publishing Co., New York, 1965.


%\bibitem[AhK62]
%{AhK62}
%	N. I. Akhiezer, M. Krein, 
%		\textit{Some questions in the theory of moments}. 
%			Transl.\ Math.\ Monographs 2, American Math.\ Soc.\, Providence, 1962.
%
%\bibitem[Alb69]
%{Alb69}
%	A.\ Albert, 
%		\textit{Conditions for positive and nonnegative definiteness in terms of pseudoinverses}. 
%			SIAM J. Appl. Math. \textbf{17} (1969), 434--440.

\bibitem
{AJK15}
 D.\ Alpay, P.E.T.\ Jorgensen, D.P.\ Kimsey, Moment problems in an
infinite number of variables, Infin.\ Dimens.\ Anal.\ Quantum Probab.\
Relat.\ Top.\ 18 (2015)
%
\bibitem
{AV03}
	C.G.\ Ambrozie, F.H.\ Vasilescu: Operator-theoretic Positivstellens\"atze,
		Z.\ Anal.\ Anwend. 22 (2003) 299--314. 

\bibitem{And70}
	T.\ Ando, Truncated moment problems for operators, Acta Sci.\ Math.\ (Szeged) 31 (1970) 319–334.

\bibitem
{BW11}
	M.\ Bakonyi, H.J.\ Woerdeman, 
		Matrix Completions, Moments, and Sums of Hermitian Squares, 
			Princeton University Press, Princeton, 2011.

\bibitem
{BT06}
	C.\ Bayer, J.\ Teichmann, The proof of Tchakaloff's theorem,
		Proc.\ Amer.\ Math.\ Soc.\ 134 (2006) 3035--3040.
%
%%\bibitem[BMV75]
%%{BMV75}
%%	Bessis, D., Moussa, P., Villani, M.: Monotonic converging variational approximations 
%%		to the functional integrals in quantum statistical mechanics.
%%		J.\ Math.\ Phys.\ 16, 2318--2325 (1975) 
%%
\bibitem
{BZ18}
	A.\ Bhardwaj, A.\ Zalar, The singular bivariate quartic tracial moment problem, Complex Anal.\ Oper.\ Theory 12:4 (2018) 1057--1142.
	\url{https://doi.org/10.1007/s11785-017-0756-3}.
%%
\bibitem
{BZ+}
	A.\ Bhardwaj, A.\ Zalar, The tracial moment problem on quadratic varieties, J.\ Math.\ Anal.\ Appl.\ 498 (2021).
	\url{https://doi.org/10.1016/j.jmaa.2021.124936}.

\bibitem
{Ble15}
	G.\ Blekherman,
 		Positive Gorenstein ideals,
			Proc.\ Amer.\ Math.\ Soc.\ 143 (2015) 69--86.
	\url{https://doi.org/10.1090/S0002-9939-2014-12253-2}.
%
\bibitem
{BF20}
	G.\ Blekherman, L.\ Fialkow, 
		The core variety and representing measures in the truncated moment problem,
			Journal of Operator Theory 84 (2020) 185--209.

%%\bibitem[Bur11]
%%{Bur11}
%%	Burgdorf, S.: 
%%		Sums of hermitian squares as an approach to the BMV conjecture.
%%		Linear and Multilinear Algebra 59, 1--9 (2011)
%

\bibitem
{BCKP13}
	S.\ Burgdorf, K.\ Cafuta, I.\ Klep, J.\ Povh,  
	The tracial moment problem and trace-optimization of polynomials,
	Math.\ Program.\ 137 (2013) 557--578.



\bibitem
{BK10}
	S.\ Burgdorf, I.\ Klep, Trace-positive polynomials and the quartic tracial moment problem,
		C.\ R.\ Math.\ Acad.\ Sci.\ Paris 348 (2010) 721--726.
	\url{https://doi.org/10.1016/j.crma.2010.06.005}.


\bibitem
{BK12}
	S.\ Burgdorf, I.\ Klep, The truncated tracial moment problem,
		J.\ Oper.\ Theory 68 (2012) 141--163. 

%%
%%\bibitem[BKP16]
%%{BKP16}
%%	Burgdorf, S., Klep, I., Povh, J.: Optimization of polynomials in non-commuting variables. 
%%		SpringerBriefs in Mathematics, Springer-Verlag (2016)
%
%%
%%\bibitem[CW18]
%%{CW18}
%%	Cao, L., Woerdeman, H.J.: Real zero polynomials and A.\ Horn's problem. Linear Algebra Appl.\ 552, 147--158 (2018)
%
\bibitem
{CH98}
 G.N.\ Chen, Y.J.\ Hu, The truncated Hamburger matrix moment problems in the nondegenerate and degenerate cases, and matrix continued
fractions, Linear Algebra Appl. 277 (1998) 199--236.
\url{https://doi.org/10.1016/S0024-3795(97)10076-3}.

\bibitem
{CZ12}
	J.\ Cimpri\v c, A.\ Zalar, Moment problems for operator polynomials,
		J.\ Math.\ Anal.\ Appl.\ 401 (2013) 307--316.
	\url{https://doi.org/10.1016/j.jmaa.2012.12.027}.
%
%%\bibitem[Con76]
%%{Con76}
%%	Connes, A.: Classification of injective factors. Cases $I\! I_1$, $I\! I_{\infty}$, 
%%	$I\! I\! I_{\lambda}$, $\lambda\neq 1$. Ann.\ Math.\ 104, 73--115 (1976) 
%
%%\bibitem[CJ16]
%%{CJ16}
%%	H. Choi, F. Jafari,
%%		\textit{Positive definite Hankel matric completions and Hamburger moment completions.}
%%			Linear algebra and its applications \textbf{489} (2016), 217--237.  
%	
%		
%
%\bibitem[CH69]
%{CH69}
%	D. Crabtree, E. Haynsworth, 
%		\textit{An identity for the Schur complement of a matrix}. 
%			Proc. Am. Math. Soc. \textbf{22} (1969), 364--366.
%
\bibitem
{CF91}
	R.\ Curto, L.\ Fialkow, 
		Recursiveness, positivity, and truncated moment problems,
			Houston J.\ Math.\ 17 (1991) 603--635.
%
%
\bibitem
{CF96}
	 R.\ Curto, L.\ Fialkow, Solution of the truncated complex moment problem
	 	for flat data, Mem.\ Amer.\ Math.\ Soc.\ 119 (1996).
%
%\bibitem[CF98a]
%{CF98a}
%	R.\ Curto, L.\ Fialkow, \textit{Flat extensions of positive moment matrices: 
%		relations in analytic or conjugate terms}.
%		Oper. Theory Adv. Appl. \textbf{104} (1998), 59--82. 
%		
\bibitem
{CF98}
	R.\ Curto, L.\ Fialkow, Flat extensions of positive moment matrices: 
		recursively generated relations, Mem.\ Amer.\ Math.\ Soc.\ 136 (1998).
%
\bibitem
{CF02}
	R.\ Curto, L.\ Fialkow, Solution of the singular quartic moment problem,
		J.\ Operator Theory 48 (2002) 315--354.

\bibitem
{CF04}
	R.\ Curto, L.\ Fialkow, Solution of the truncated parabolic moment problem,
		Integral Equations Operator Theory 50  (2004) 169--196.
		\url{https://doi.org/10.1007/s00020-003-1275-3}.
		
\bibitem
{CF05}
	R.\ Curto, L.\ Fialkow, Solution of the truncated hyperbolic moment problem,
		Integral Equations Operator Theory 52 (2005) 181--218.  
		\url{https://doi.org/10.1007/s00020-004-1340-6}.

\bibitem
{CF05b}
	R.\ Curto, L.\ Fialkow, Truncated $K$-moment problems in several variables,
		J.\ Operator Theory 54 (2005) 189--226.
%		
%\bibitem[CF08]
%{CF08}
%	R.\ Curto, L.\ Fialkow, \textit{An analogue of the Riesz-Haviland theorem for the truncated moment 
%		problem}. J. Funct. Anal. \textbf{225} (2008), 2709--2731.
%
%
\bibitem
{CF13}
	R.\ Curto, L.\ Fialkow, Recursively determined representing measures for bivariate
		truncated moment sequences, J. Operator theory 70 (2013) 401--436.

\bibitem
{CFM08}
	R.\ Curto, L.\ Fialkow, H. M. M\"oller, The extremal truncated moment problem,
		Integral Equations Operator Theory 60 (2) (2008) 177-200. 
	\url{https://doi.org/10.1007/s00020-008-1557-x}.

\bibitem
{CGIK+}
	R.\ Curto, M. Ghasemi, M. Infusino, S. Kuhlmann, The truncated moment problems for unital commutative $\RR$-algebras, 
	arxiv preprint \url{https://arxiv.org/pdf/2009.05115.pdf}. 


\bibitem
{CS15}
	R.\ Curto, S.\ Yoo, Non-extremal sextic moment problems, J.\ Funct.\ Anal.\ 269 (3) (2015) 758--780. 
	\url{https://doi.org/10.1016/j.jfa.2015.04.014}.

\bibitem
{CS16}
	R.\ Curto, S.\ Yoo, Concrete solution to the nonsingular quartic binary moment problem,
		Proc.\ Amer.\ Math.\ Soc.\ 144 (2016) 249--258. 
	\url{https://doi.org/10.1090/proc/12698}.
%
%\bibitem[Dan92]
%{Dan92}
%	J. Dancis, \textit{Positive semidefinite completions of partial hermitian matrices}.
%		Linear Algebra Appl. \textbf{175} (1992), 97--114.
%
%%\bibitem[DP01]
%%{DP01}
%%	C. N. Delzell, A. Prestel, 
%%	\textit{Positive polynomials. From Hilbert's 17th problem to real algebra}.
%%	Springer Monogr. Math., (2001)
%
%%\bibitem[DLTW08]
%%{DLTW08}
%%	Doherty, A.C., Liang, Y.-C., Toner, B., Wehner, S.:
%%		The quantum moment problem and bounds on entangled multi-prover games.
%%		In: Twenty-Third Annual IEEE Conference on Computational Complexity, pp.\ 199--210. 
%%		IEEE Computer Soc., Los Alamitos, CA (2008)
%
%\begin{color}{blue}
%\bibitem[dDio19]
%P.J. di Dio, \textit{The multidimensional truncated moment problem: Gaussian and log-normal mixtures, their Caratheodory numbers, and set of atoms}, 
%	Proc. Amer. Math. Soc. \textbf{147} (2019), 3021--3038.
%\end{color}


\bibitem
{DS18}
 	P.J.\ di Dio, K.\ Schm\"udgen, The multidimensional truncated Moment Problem: Atoms, Determinacy, and Core Variety,
	J.\ Funct.\ Anal.\ 274 (2018) 3124--3148.
	\url{https://doi.org/10.1016/j.jfa.2017.11.013}.

\bibitem
{Dym89}
 	H.\ Dym, On Hermitian block Hankel matrices, matrix polynomials, the
		Hamburger moment problem, interpolation and maximum entropy, Integral Equations and Operator Theory 12 (1989) 757--812.


\bibitem
{Ers68}
 	V.G.\ Ershov, The truncated matrix power moment problem, Izv.\
	Vyssh.\ Uchebn.\ Zaved.\ Mat., Matematika 76 (1968) 36--43. (in Russian)
%
\bibitem
{FN10}
	L.\ Fialkow, J.\ Nie,  
		Positivity of Riesz functionals and solutions of quadratic and quartic
		moment problems, J.\ Funct.\ An.\ 258 (2010) 328--356.
%
%
%\bibitem[Fia08]
%{Fia08}
%	L.\ Fialkow, \textit{Truncated multivariable moment problems with finite variety}. J. of Operator Theory \textbf{60} (2008), 343--377.

\bibitem
{Fia11}
	L.\ Fialkow, Solution of the truncated moment problem with variety $y=x^3$,
	Trans.\ Amer.\ Math.\ Soc.\ 363 (2011) 3133--3165. 

\bibitem
{Fia14}
	L.\ Fialkow, The truncated moment problem on parallel lines, 
	The Varied Landscape of Operator Theory (2014) 99--116.

\bibitem
{Fia17}
	L.\ Fialkow, The core variety of a multisequence in the truncated moment problem,
	J.\ Math.\ Anal.\ Appl.\ 456 (2017) 946--969. 
	\url{https://doi.org/10.1016/j.jmaa.2017.07.041}.

\bibitem
{GKM16}
	M.\ Ghasemi, S.\ Kuhlmann, M.\ Marshall, Moment problem in infinitely many variables, Israel J.\ Math 212 (2016) 989--1012.


%\bibitem[GJSW84]
%{GJSW84}
%	R. Grone, C. R. Johnson, E. M. S\'a, H. Wolkowicz, \textit{Positive definite completions of partial hermitian matrices}.
%	Linear Algebra Appl. \textbf{58} (1984), 109--124. 
%
\bibitem
{Hav35}
	E.K.\ Haviland, On the momentum problem for distribution functions in more than one dimension II, Amer.\ J.\ Math.\ 58 (2006) 164--168.
%	
%%\bibitem[Hel02]
%%{Hel02}
%%	Helton, J.W.: ``Positive" noncommutative polynomials are sums of squares.
%%		Ann.\ of Math.\ 156, 675--694 (2002) 
%
\bibitem
{HKM12} 
	J.W.\ Helton, I.\ Klep, S.\ McCullough, The convex Positivstellensatz in a free algebra,
		Adv.\ Math.\ 231 (2017) 516--534.
%%
\bibitem
{HM04}
	J.W.\ Helton, S.\ McCullough, A Positivstellensatz for noncommutative polynomials,
	Trans.\ Amer.\ Math.\ Soc.\ 365 (2004) 3721--3737.



\bibitem{IKLS17}
	M.\ Infusino, T.\ Kuna, J.L.\ Lebowitz, E.R.\ Speer, The truncated moment problem on $\NN_0$, J.\ Math.\ Anal.\ Appl.\ 452 (2017) 443--468.
	\url{https://doi.org/10.1016/j.jmaa.2017.02.060}.

%\begin{color}{blue}
%\bibitem[Ioh82]
%{Ioh82}
%	I. S. Iohvidov, \textit{Hankel and Toeplitz matrices and forms: Algebraic theory}.
%		Birkh\"auser Verlag, Boston, 1982.
%\end{color}
 
\bibitem
{Kim14}
	D.P.\ Kimsey, The cubic complex moment problem, Integral Equations Operator Theory 80 (2014) 353–378.
	\url{https://doi.org/10.1007/s00020-014-2183-4}.
% 
\bibitem
{Kim+}
	D.P.\ Kimsey, On a minimal solution for the indefinite truncated multidimensional moment problem,
	J.\ Math.\ Anal.\ Appl.\ 500 (2021).
	\url{https://doi.org/10.1016/j.jmaa.2021.125091}.
	 

\bibitem
{KP+}
	D.P.\ Kimsey, M.\ Putinar, The moment problem on curves with bumps,
	Math.\ Z.\ 298, 935--942 (2021).
	\url{https://doi.org/10.1007/s00209-020-02633-2}.

\bibitem
{KW13}
	D.P.\ Kimsey, H.\ Woerdeman, The multivariable matrix valued $K$-moment problem
	on $\RR^d$, $\CC^d$, $\mathbb{T}^d$,
	Trans. Amer. Math. Soc. 365 (2013) 5393--5430.
	\url{https://doi.org/10.1090/S0002-9947-2013-05812-6}.

%
\bibitem
{Kre49}
	M.G.\ Krein, Infinite $J$-matrices and a matrix-moment problem, Doklady Akad.\ Nauk SSSR 69 (1949) 455-497.

\bibitem
{KN77}
	M.G.\ Krein, A.A.\ Nudelman, The Markov moment problem and extremal problems,
	Translations of Mathematical Monographs, Amer. Math. Soc., 1977.

%
%%\bibitem[KS08a]
%%{KS08-1}
%%	Klep, I., Schweighofer, M.: Connes' embedding conjecture and sums of hermitian squares.
%%	Adv.\ Math.\ 217, 1816--1837 (2008) 
%%
%%\bibitem[KS08b]
%%{KS08-2}
%%	Klep, I., Schweighofer, M.: Sums of hermitian squares and the BMV conjecture.
%%		J.\ Stat.\ Phys.\ 133, 739--760 (2008)
%

\bibitem
{KM02}
	S.\ Kuhlmann, M.\ Marshall, Positivity, sums of squares and the multidimensional moment problem, 	
		Trans.\ Amer.\ Math.\ Soc.\ 354 (2002) 4285--4301.


\bibitem
{Las01}
	J.B.\ Lasserre, Global optimization with polynomials and the problem of moments,
		SIAM J.\ Optim.\ 11 (3) (2001) 796--817.

%
\bibitem
{Las09}
	J.B.\ Lasserre, Moments, positive polynomials and their applications,
		Imperial College Press, 2009.


\bibitem
{Lau05}
	M.\ Laurent, 
	Revising two theorems of Curto and Fialkow on moment matrices,
			Proc.\ Amer.\ Math.\ Soc.\ 133 (2005) 2965--2976. 


\bibitem
{Lau09}
	M.\ Laurent, Sums of squares, moment matrices and optimization over polynomials,
	In: Emerging Applications of Algebraic Geometry, Vol. 149 of IMA Volumes in Mathematics and its 
	Applications, pp.\ 157--270,  Springer-Verlag, 2009.

%
%%\bibitem[LP15]{LP15}
%% 	Laurent, M., Piovesan, T.: Conic approach to quantum graph parameters using linear optimization over the completely positive semidefinite cone. 
%%	SIAM J.\ Optim.\ 25, 2461--2493 (2015)
%
% 
%

\bibitem
{Mar08}
	M.\ Marshall, Positive polynomials and sums of squares,
		Mathematical Surveys and Monographs 146, 
			Amer.\ Math.\ Soc., 2008.

%%
\bibitem
{McC01}
	S.\ McCullough, Factorization of operator-valued polynomials in several non-commuting variables,
	Linear Algebra Appl.\ 326 (2001) 193--204.
	\url{https://doi.org/10.1016/S0024-3795(00)00285-8}.
%
%%\bibitem[MLH11]
%%{MLH11}
%%	M. Mevissen, J. B. Lasserre, D. Henrion, 
%%	\textit{Moment and SDP relaxation techniques for smooth approximations of problems involving nonlinear differential equations}.
%%		IFAC Proceedings Volumes, \textbf{44 (1)} 2011, 10887--10892.
%


\bibitem{Nie14}
	J.\ Nie, The $\mathcal{A}$-truncated $K$-moment problem, 
		Found.\ Comput.\ Math.\ 14 (2014) 1243--1276. 


%%\bibitem[Par03]
%%{Par03}
%%	P. Parrilo, \textit{Semidefinite programming relaxations for semialgebraic problems}. 
%%		Mathematical programming \textbf{96 (2)} (2003), 293--320.
%
%
\bibitem
{PS01}
	V.\ Powers, C.\ Scheiderer, The moment problem for non-compact semialgebraic sets,
	Adv.\ Geom.\ 1 (2001) 71--88. 

\bibitem
{Put93}
	M.\ Putinar, Positive polynomials on compact semi-algebraic sets,
	Indiana Univ.\ Math.\ J.\ 42 (1993) 969--984. 
%
%
%
%
%\begin{color}{blue}
%\bibitem[PS06]
%{PS06}
%	M. Putinar, C. Scheiderer,  
%		\textit{Multivariate moment problems: Geometry and indeterminateness}.
%			Ann. Sc. Norm. Super. Pisa Cl. Sci. \textbf{5} (2006), 137--157.
%\end{color}
%
%
%\bibitem[PS08]
%{PS08}
% 	M. Putinar, K. Schm\"udgen, 
%		\textit{Multivariate determinateness}. 
%			Indiana Univ. Math. J. \textbf{57} (2008), 2931--2968.



\bibitem
{PV99}
	M.\ Putinar, F.H.\ Vasilescu, Solving moment problems by dimensional extension,
		Ann.\ of Math.\ 149  (1999) 1087--1107.

%
%%\bibitem[Qua15]
%%{Qua15}
%%	Quarez, R.: Trace-positive non-commutative polynomials.
%%		Proc.\ Amer.\ Math.\ Soc. 143, 3357--3370 (2015)
%

\bibitem{Sch91}
	K.\ Schm\"udgen, The K-moment problem for compact semi-algebraic sets,
		Math.\ Ann.\ 289 (1991) 203--206.

%\begin{color}{blue}
%\bibitem[Sch03]
%{Sch03}
%	K. Schm\"udgen, \textit{On the moment problem for closed semi-algebraic sets}.
%		J. Reine Angew. Math. \textbf{588} (2003), 225--234. 
%\end{color}
%
\bibitem
{Sch17}
	K.\ Schm\"udgen, The moment problem,
		Graduate Texts in Mathematics 277, Springer, Cham, 2017.

%
\bibitem
{Sim06}
	K.K.\ Simonov, Strong truncated matrix moment problem of Hamburger,
		Sarajevo J.\ Math.\ 2(15) (2006), no.\ 2, 181--204.

%
%\bibitem
%{Smu59}
%	 J.L.\ Smul'jan, An operator Hellinger integral, Mat.\ Sb.\ (N.S.) 49 (1959) 381--430.

%%\bibitem[Sta13]
%%{Sta13}
%%	Stahl, H.R.: Proof of the BMV conjecture, Acta Mathematica 211 (2), 255-290 (2013) 
%
%

\bibitem
{Sto01}
	J.\ Stochel, Solving the truncated moment problem solves the moment problem,
	Glasgow J.\ Math.\ 43 (2001) 335--341. 


%
%\begin{color}{blue}
%\bibitem[Tch57]
%{Tch57}
%	Tchakaloff, V.: 
%	Formules de cubatures m\'ecaniques \`a coefficients non
%	n\'egatifs. Bull.\ Sci.\ Math.\ 81, 123--134 (1957) 
%\end{color}

\bibitem
{Vas03}
	F.H.\ Vasilescu, Spectral measures and moment problems,
	In: Spectral theory and its applications (2003) 173--215.

% 
%\bibitem[Wol]
%{Wol}
%	Wolfram Research, Inc., Mathematica, Version 10.0, Wolfram Research, Inc., Champaign,
%		IL, 2019.

%%\bibitem[Yoo11]
%%{Yoo11}
%%	Yoo, S.: Extremal sextic truncated moment problems, (2011)
%
\bibitem
{Zal+}
	A.\ Zalar, The truncated Hamburger moment problem with gaps in the index set, 
	Integral Equations Operator Theory 93 (2021). 
	\url{https://doi.org/10.1007/s00020-021-02628-6}.
%
\bibitem%[Zha05]
{Zha05}
	F.\ Zhang, The Schur Complement and Its Applications. 
	Springer-Verlag, New York, 2005.
%
%
%
%


\end{thebibliography}


\end{document}